\documentclass[12pt]{elsarticle}
\usepackage[all,cmtip]{xy}
\usepackage{hyperref}
\hypersetup{colorlinks = true, allcolors = blue}
\usepackage{fullpage}

\usepackage{mathtools}
\usepackage{extarrows}

\usepackage[utf8]{inputenc}
\usepackage{setspace}
\usepackage{tocloft}

\biboptions{square,sort,comma,numbers}
\usepackage{graphicx, titlesec}
\usepackage{amscd,amsmath,amsthm,amssymb,bm}%\usepackage{amsmath,bm}

\usepackage{verbatim}
\usepackage[all]{xy}
\usepackage{color}
\usepackage{tikz, float} \usetikzlibrary {positioning}
\usetikzlibrary{calc}

\usepackage{makeidx}         % allows index generation
\usepackage{graphicx}        % standard LaTeX graphics tool
\usepackage{multicol}        % used for the two-column index
\usepackage[bottom]{footmisc}% places footnotes at page bottom

\usepackage{array}

\usepackage{newtxmath}      % provides definitions for: \jmath, \amalg, \coprod
\usepackage{theoremref}
\usepackage{url}

\makeindex

\newcommand{\BLambda}{{B_\ell}}
\newcommand{\id}{\text{id}}
\newcommand{\ini}{\text{in}}

\newcommand{\MT}{\mathcal T}

\newcommand{\KK}{\mathbb{K}}

\titleformat{\subsection}[runin]
{\normalfont\bfseries}{\thesubsection}{.5em}{}
\def\NZQ{\mathbb}               % the font for N,Z,Q,R,C

\def\RR{{\NZQ R}}

\def\AA{{\NZQ A}}
\def\PP{{\NZQ P}}

\def\frk{\mathfrak}               % font for "Fraktur"

\def\Phi{{\frk n}}
\def\Phi{{\frk N}}

\def\wb{{\mathbf w}}

\def\A{{\mathcal A}}

\def\opn#1#2{\def#1{\operatorname{#2}}} % to make operators
\opn\chara{char} \opn\length{\ell} \opn\pd{pd} \opn\rk{rk}
\opn\projdim{proj\,dim} \opn\injdim{inj\,dim} \opn\rank{rank}
\opn\depth{depth} \opn\grade{grade} \opn\height{height}
\opn\embdim{emb\,dim} \opn\codim{codim}

\opn\Tr{Tr} \opn\bigrank{big\,rank}
\opn\superheight{superheight}\opn\lcm{lcm}
\opn\trdeg{tr\,deg}%\emph{
\opn\reg{reg} \opn\lreg{lreg} \opn\ini{in} \opn\lpd{lpd}
\opn\size{size} \opn\sdepth{sdepth}
\opn\link{link}\opn\fdepth{fdepth}\opn\lex{lex}
\opn\LM{LM}
\opn\LC{LC}
\opn\NF{NF}
\opn\Merge{Merge}
\opn\sgn{sgn}
\opn\type{type}
\opn\div{div} \opn\Div{Div} \opn\cl{cl} \opn\Pic{Pic}
\opn\Prin{Prin}
\opn\op{op}
\opn\indeg{indeg} \opn\outdeg{outdeg}
\opn\red{red}
\opn\Spec{Spec} \opn\Supp{Supp} \opn\supp{supp} \opn\Sing{Sing}
\opn\Ass{Ass} \opn\Min{Min}\opn\Mon{Mon} \opn\val{val}
\opn\Ann{Ann} \opn\Rad{Rad} \opn\Soc{Soc}
 \opn\Ker{Ker} \opn\Coker{Coker} \opn\Am{Am}
\opn\Hom{Hom} \opn\Tor{Tor} \opn\Ext{Ext} \opn\End{End}
\opn\Aut{Aut} \opn\id{id}

\opn\nat{nat}
\opn\pff{pf}%   \pf exists already
\opn\Pf{Pf} \opn\GL{GL} \opn\SL{SL} \opn\mod{mod} \opn\ord{ord}
\opn\Gin{Gin} \opn\Hilb{Hilb}\opn\sort{sort}
\opn\Image{Image}
\opn\vol{Vol}
\opn\aff{aff} \opn\con{conv} \opn\relint{relint} \opn\st{st}
\opn\lk{lk} \opn\cn{cn} \opn\core{core} \opn\vol{vol}
\opn\link{link} \opn\star{star}\opn\lex{lex}\opn\set{set}
\opn\dist{dist}
\opn\gr{gr}

\def\pot#1#2{#1[\kern-0.28ex[#2]\kern-0.28ex]}

\opn\dirlim{\underrightarrow{\lim}}
\opn\inivlim{\underleftarrow{\lim}}
\def\Implies{\ifmmode\Longrightarrow \else
        \unskip${}\Longrightarrow{}$\ignorespaces\fi}
\def\implies{\ifmmode\Rightarrow \else
        \unskip${}\Rightarrow{}$\ignorespaces\fi}
\def\iff{\ifmmode\Longleftrightarrow \else
        \unskip${}\Longleftrightarrow{}$\ignorespaces\fi}

\let\:=\colon

\newtheorem{theorem}{Theorem}[section]
\newtheorem{lemma}[theorem]{Lemma}

\newtheorem{proposition}[theorem]{Proposition}

\newtheorem{conjecture}[theorem]{Conjecture}
\newtheorem{question}[theorem]{Question}
\theoremstyle{remark}
\newtheorem{remark}[theorem]{Remark}

\theoremstyle{definition}
\newtheorem{example}[theorem]{Example}

\newtheorem{definition}[theorem]{Definition}
\newtheorem{Notation}{Notation}[section]
\newtheorem*{theorem*}{Theorem}
\newtheorem*{corollary*}{Corollary}
\DeclareMathOperator{\Gr}{Gr}

\let\kappa=\varkappa

\def\qed{\ifhmode\textqed\fi
      \ifmmode\ifinner\quad\qedsymbol\else\dispqed\fi\fi}
\def\textqed{\unskip\nobreak\penalty50
       \hskip2em\hbox{}\nobreak\hfil\qedsymbol
       \parfillskip=0pt \finalhyphendemerits=0}
\def\dispqed{\rlap{\qquad\qedsymbol}}

\opn\dis{dis}
\def\pnt{{\raise0.5mm\hbox{\large\bf.}}}

\opn\Lex{Lex}
\opn\syz{{\rm syz}}
\opn\spoly{{\rm spoly}}
\opn\LM{{\rm LM}}
\opn\lm{{\rm lm}}
\opn\lcm{{\rm lcm}} \opn\A{\mathcal A}

\numberwithin{equation}{section}

\newcommand{\inwb}{{\rm in}_{{\bf w}_\ell}}

\DeclareMathOperator{\init}{in}

\usepackage{multirow}
\hypersetup{
  colorlinks=true,
}
\begin{document}

\begin{frontmatter}

\title{Standard monomial theory and toric degenerations of\\
Richardson varieties in the Grassmannian} 
\author{
Narasimha Chary Bonala, Oliver Clarke and Fatemeh Mohammadi
}

\begin{abstract}
Richardson varieties
are obtained as intersections of Schubert and opposite Schubert varieties.
We provide a new family of toric degenerations of 
Richardson varieties inside Grassmannians by studying Gr\"obner degenerations of their corresponding ideals. These degenerations are parametrised by block diagonal matching fields in the sense of Sturmfels-Zelevinsky \cite{sturmfels1993maximal}.
We associate a weight vector to each block diagonal matching field and study its corresponding initial ideal. 
In particular, we characterise when such ideals are toric, hence providing a family of toric degenerations for Richardson varieties. 

Given a Richardson variety $X_{w}^v$ and a weight vector ${\bf w}_\ell$ arising from a matching field, we consider two ideals: an ideal $G_{k,n,\ell}|_w^v$ obtained by restricting the initial of the Pl\"ucker ideal to a smaller polynomial ring, and a toric ideal defined as the kernel of a monomial map $\phi_\ell|_w^v$. We first characterise the monomial-free ideals of form $G_{k,n,\ell}|_w^v$. Then we construct a family of tableaux in bijection with semi-standard  Young  tableaux which leads to a monomial basis for the corresponding quotient ring. 
Finally, we prove that when $G_{k,n,\ell}|_w^v$ is monomial-free and the initial ideal in$_{{\bf w}_\ell}(I(X_w^v))$ is quadratically generated, then all three ideals in$_{{\bf w}_\ell}(I(X_w^v))$, $G_{k,n,\ell}|_w^v$ and ker$(\phi_\ell|_w^v)$ coincide, and provide a toric degeneration of 
$X_w^v$. 
\end{abstract}
\end{frontmatter}
\vspace{-2mm}

{\hypersetup{linkcolor=black}\setcounter{tocdepth}{1}\setlength\cftbeforesecskip{.4pt}{\tableofcontents}}

\section{Introduction}\label{sec:intro}
Toric degenerations are a particularly useful tool for describing algebraic properties of varieties in terms of combinatorics of polytopes and polyhedral fans, see  \cite{anderson2013okounkov}. 
The goal of this paper is to construct a family of toric degenerations for Richardson varieties inside the Grassmannian. To do this, we will consider a family of matching fields, which were originally introduced by Sturmfels and Zelevinsky for studying certain Newton polytopes \cite{sturmfels1993maximal}. We study Gr\"obner degenerations of the Pl\"ucker ideals associated to matching fields.

Toric degenerations of Grassmannians and Schubert varieties have been extensively studied in the literature, see e.g.~\cite{gonciulea1996degenerations,caldero2002toric,fang2017toric,speyer2004tropical, rietsch2017newton, BFFHL}. In particular, most of these degenerations can be realised as Gr\"obner degenerations, even though, this is not true in general, see e.g.~\cite{kateri2015family}.
More recently, Kaveh and Manon \cite{kaveh2019khovanskii} used tools from tropical geometry to study toric degenerations of ideals. The main idea is that the initial ideals associated to the points in the top-dimensional facets of tropicalizations of ideals are good candidates for toric degenerations. A similar approach has been taken in studying toric degenerations of Grassmannians and Schubert varieties in \cite{KristinFatemeh, OllieFatemeh, kogan, bossinger2020families, serhiyenko2019cluster}.

Let $k \le n$ be positive integers and let ${\bf I}_{k,n} = \{I \subseteq [n] : |I| = k \}$ be the set of $k$-subsets of $[n] := \{1, \dots, n \}$. The Grassmannian $\Gr(k,n)$ is the collection of $k$-dimensional linear subspaces of $\KK^n$. By the Pl\"ucker embedding, $\Gr(k,n)$ is viewed as a projective subvariety of $\PP^{\binom n k - 1}$ whose ideal is denoted by $G_{k,n} \subseteq \KK[P_I]_{I \in {\bf I}_{k,n}}$. 
The intersection of Schubert and opposite Schubert varieties in $\Gr(k,n)$ give rise to Richardson varieties which are indexed by pairs of subsets $v,w \in {\bf I}_{k,n}$ where $v \le w$.
An explicit description of the ideal of a Richardson variety is given by Kreiman and Lakshmibai \cite{kreiman2002richardson}. We view this ideal as the \textit{restriction} of $G_{k,n}$ to the variables $P_I$ where $I$ is in the set
$T^v_{w}=\{I \in {\bf I}_{k,n} : v \le I \le w\}$. More specifically, the ideal of the Richardson variety $X_w^v$ is defined as
\[
I(X_w^v) := G_{k,n}|_{T_{w}^v} = \left(G_{k,n} + \langle P_I: I\in {\bf I}_{k,n}\backslash T_{w}^v\rangle \right) \cap \KK[P_I]_{I\in T_{w}^v}
.
\]
Our goal is to find toric degenerations of the Richardson variety $X_w^v$ by studying its Gr\"obner degenerations. We consider a collection of weight vectors ${\bf w}_\ell \in \RR^{\binom n k}$ associated to so-called \textit{block diagonal matching fields}, see Definition~\ref{example:block} and Remark~\ref{def:block}. In \cite{OllieFatemeh}, it is shown that each ideal $\init_{{\bf w}_\ell}(G_{k,n})$ is  toric and equals to the kernel of a monomial map $\phi_\ell$, 
see Equation~\eqref{eqn:monomialmap} and Theorem~\ref{thm:JAlebra}. 
In particular, each vector ${\bf w}_\ell$ gives rise to a flat family over $\AA^1$ whose fiber over $0$ is given by the initial ideal $\init_{{\bf w}_\ell}(G_{k,n})$, see e.g.~\cite[Theorem~15.17]{eisenbud2013commutative}. 
Here, for the Richardson variety we project the weight vector ${\bf w}_\ell$ to the coordinates corresponding to the variables in the polynomial ring 
$\KK[P_I]_{I\in T_w^v}$
and study the 
following question. Note that this is related to {\it Degeneration Problem} posed by Caldero \cite{caldero2002toric} for Schubert varieties.
\begin{question}\label{intro:question}
When are the initial ideals $\init_{\bf{w}_\ell}(I(X_w^v))$ toric (binomial and prime)?
\end{question}
We have summarised our approach to this question in the following diagram. 
\begin{equation}\label{diagram:intro_commutative}
\begin{aligned}
\xymatrixcolsep{3.5pc}
 \xymatrix{
    {G_{k,n}} 
    \ar@{~>}[r]^-{\text{initial ideal}}
    \ar@{~>}[d]_{\text{restriction}} &
    {\init_{{\bf w}_\ell}(G_{k,n})} 
    \ar@{~>}[d]^{\text{restriction}}
    \ar@2{-}[r]^-{\eqref{eq:JAlgebra}} &
    {\ker(\phi_\ell) 
    \subseteq
    \KK[P_I : I \in {\bf I}_{k,n}] 
    \xlongrightarrow{\phi_\ell}
    \KK[x_{i,j}]}
    \ar@{~>}[d]^-{\text{restriction of $\phi_\ell$}}
    \\
    {I(X_w^v)} 
    \ar@{~>}[r]_-{\text{initial ideal}}
    & 
    {G_{k,n,\ell}|_{w}^v}
    \ar@2{-}[r]^-{\textcolor{blue}{{\bm ?}}} &
    {\ker(\phi_\ell|_w^v)
    \subseteq
    \KK[P_I : I \in T^v_{w}] 
    \xlongrightarrow{\phi_\ell|_w^v}
    \KK[x_{i,j}]} \\
    }
\end{aligned}
\end{equation}
A typical approach to study the ideal $\init_{\bf{w}_\ell}(I(X_w^v))$ is to search for a Gr\"obner basis which can be viewed as understanding the horizontal arrows in the above diagram. By \cite{OllieFatemeh}, we have a clear description of the ideals in the first row of the diagram and their interrelations, namely $\init_{\bf{w}_\ell}(G_{k,n}) = \ker(\phi_\ell)$. To find toric initial ideals $\init_{\bf{w}_\ell}(I(X_w^v))$ we will determine when the left-hand square is commutative and when `${\bm?}$' is an equality. To do so, we first study the ideals $G_{k,n,\ell}|_w^v$ obtained by restricting $\init_{{\bf w}_\ell}(G_{k,n})$ to the variables $P_I$, where $I\in T_{w}^v$. We characterise such monomial-free ideals and show that in this case, $\init_{\bf{w}_\ell}(I(X_w^v))$ is equal to $G_{k,n,\ell}|_w^v$. 
Here, $\KK[x_{i,j}]$ is the polynomial ring in the indeterminates $x_{i,j}$ for $i\in[k]$ and $j\in[n]$. 

\begin{theorem*}[Theorem~\ref{thm:main}]\label{thm:main_intro}
Let $\bf{w}_\ell$ be the weight vector induced by a block diagonal matching field for $\Gr(k,n)$. If the ideal $G_{k,n,\ell}|^v_w$ is monomial-free and 
$\init_{\bf{w}_\ell}(I(X_w^v))$ is generated by degree two polynomials, 
then Diagram~(\ref{diagram:intro_commutative}) commutes. Explicitly the ideals $\init_{\bf{w}_\ell}(I(X_w^v))$, $G_{k,n,\ell}|^v_w$ and $\ker(\phi_\ell|_w^v)$ are all equal and provide a toric degeneration for the Richardson variety $X_w^v$.
\end{theorem*}
Our strategy for the proof is to consider the following inclusions which always hold:
\begin{equation}\label{eq:inclusion}
 \xymatrixcolsep{5pc} 
 \xymatrix{
    \ker(\phi_\ell|^v_w)
    \ar@{^{(}->}[r]
    &G_{k,n,\ell}|^v_w
    \ar@{^{(}->}[r]
    & \init_{{\bf w}_\ell}(I(X_w^v))
    }.
\end{equation}
We proceed by classifying the restricted ideals  $G_{k,n,\ell}|^v_w$  which are monomial-free. For each block diagonal matching field $B_\ell$, we give combinatorial conditions on the permutations $v$ and $w$ for which $G_{k,n,\ell}|^v_w$ is monomial-free, see Theorem~\ref{thm:Rich}. Surprisingly, the conditions on $v$ and $w$ are independent and the  ideal $G_{k,n,\ell}|^v_w$ is monomial-free if and only if the ideals of the corresponding Schubert and opposite Schubert varieties are monomial-free. 

Our main tool to prove the equality of $\init_{\bf{w}_\ell}(I(X_w^v))$ and  $\ker(\phi_\ell|_w^v)$ 
is to apply the above classification to describe a \textit{monomial basis} for the quotient ring $\KK[P_I]_{I\in T_{w}^v}/\ker(\phi_\ell|^v_w)$.  
We use the Hodge's description \cite{hodge1943some} for \textit{standard} monomial basis of $\KK[P_I]_{I\in \mathbf{I}_{k,n}}/G_{k,n}$ which is in terms of semi-standard Young tableaux. We note that this basis is compatible with any Richardson variety $X_w^v$, i.e.~the basis elements that remain non-zero after restricting to $X_w^v$ form a basis for the coordinate ring associated of $X_w^v$. 
More specifically, we first prove that:

\begin{theorem*}[Proposition~\ref{lem:SSYT_Gamma_surj_restrict} and Theorem~\ref{thm:std_monomials}]
\label{thm:intro_std_monomials}
If $G_{k,n,\ell}|_w^v$ is monomial-free, then the size of a monomial basis in degree two of $\ker(\phi_\ell|_w^v)$ is equal to the number of semi-standard Young tableaux with two columns $I, J$ such that $v \le I, J \le w$. 
\end{theorem*}
\vspace{-2mm}
Then, we show that when $G_{k,n,\ell}|^v_w$ is monomial-free and $\init_{{\bf w}_\ell}(I(X_w^v))$ is quadratically generated, then the number of standard monomials for $\KK[P_I]_{I\in T_{w}^v}/\ker(\phi_\ell|^v_w)$ equals to the number of semi-standard Young tableaux.
In particular, the dimensions of the quotient rings of $\init_{\bf{w}_\ell}(I(X_w^v))$ and  $\ker(\phi_\ell|_w^v)$ are equal in each degree. This together with the inclusion in \eqref{eq:inclusion} implies that these ideals are equal.

\medskip

We note that the assumption that $\init_{{\bf w}_\ell}(I(X_w^v))$ is quadratically generated is crucial for our proof. However, we expect that whenever $G_{k,n,\ell}|_w^v$ is monomial-free then the ideal $\init_{\bf{w}_\ell}(I(X_w^v))$ is quadratically generated and we may remove this assumption.

\begin{conjecture}\label{conj:J_2_quad_gen}
If $G_{k,n,\ell}|_w^v$ is monomial-free then $\init_{\bf{w}_\ell}(I(X_w^v))$ is quadratically generated.
\end{conjecture}
In general, showing that an ideal is quadratically generated is a challenging task, see e.g.~\cite{hibi1987distributive,ene2011monomial,Ohsugi,Cone} for specific families of ideals arising from graphs, and \cite{KristinFatemeh,OllieFatemeh,CharyOllieFatemeh_WICA} for similar families of initial ideals associated to matching fields.
We prove Conjecture~\ref{conj:J_2_quad_gen} for $\ell=0$ in Theorem~\ref{thm:intro_ell_0} and provide computational evidence for general $\ell>0$ in Remark~\ref{rem:computation} and \cite{OllieCodeGr}.

\smallskip
\noindent{\bf Outline of the paper.}
In Section~\ref{sec:prim} we fix our notation and recall the definitions of the main objects under study such as Richardson varieties, matching fields and their associated ideals. 
Section~\ref{sec:gr} contains our main results characterising monomial-free ideals of form $G_{k,n,\ell}|_w^v$, see Theorem~\ref{thm:Rich}. 
In Section~\ref{sec:standard_monomial} we study monomial bases of Richardson varieties, see Theorem~\ref{thm:std_monomials}.
Section~\ref{sec:main_proofs} contains the proofs of our main results stated in the introduction, see Theorems~\ref{thm:main}.

\smallskip
\noindent{\bf Acknowledgement.} NC was supported by the SFB/TRR 191 ``Symplectic structures in Geometry, Algebra and Dynamics".
He gratefully acknowledges support from the Max Planck Institute for Mathematics in Bonn, and the EPSRC Fellowship EP/R023379/1 who supported his multiple visits to Bristol. 
OC was supported by EPSRC Doctoral Training Partnership award EP/N509619/1.
FM was partially supported by EPSRC Early Career Fellowship EP/R023379/1, the BOF Starting Grant of Ghent University STA/201909/038 and FWO grants (G023721N and G0F5921N). This project began during the ``Workshop for Young Researchers"
in Cologne. We would like to thank the organisers of the meeting, Lara Bossinger and Sara Lamboglia, and in particular, Stephane Launois for supporting the author's visit via the EPSRC grant EP/R009279/1.

\vspace{-2mm}
\section{Preliminaries}\label{sec:prim}
\vspace{-2mm}
Throughout we fix an algebraically closed field $\KK$. 
We let $[n]$ be the set $\{1, \dots, n \}$, ${\bf I}_{k,n}$ be the collection of $k$-subsets of $[n]$ and
$\KK[x_{i,j}]$ be the polynomial ring on the variables $x_{i,j}$ with $i \in [k]$ and $j \in [n]$. We first recall 
Richardson varieties and their corresponding standard monomial bases.
Then, we define matching fields and their associated ideals in \S\ref{sec:prelim_matching_field} and \S\ref{sec:ideals}.

\vspace{-2mm}
\subsection{Grassmannian varieties.}\label{sec:prelim_ideals}
The Grassmannian $\Gr(k,n)$ is the collection of all $k$-dimensional linear subspaces of $\KK^n$, which is embedded into a projective space 
as follows. Any point in the Grassmannian is the rowspan of some $k \times n$ matrix and two $k \times n$ matrices and have the same rowspan if and only if they have the same sequence of maximal minors up to a scalar. And so we obtain an embedding of $\Gr(k,n)$ into a projective space
\begin{eqnarray}\label{eq:embedding}
\Gr(k,n) \rightarrow \PP^{\binom n k - 1} \quad{\rm with}\quad X = (x_{i,j}) \mapsto (\det(X_I)),
\end{eqnarray}
where $X_I$ is the submatrix of $X$ with columns indexed by $I$. The Pl\"ucker embedding gives rise to a defining ideal for the Grassmannian. Consider the map
\begin{eqnarray}\label{eq:map}
\varphi_n:\ \KK[P_I]_{I \in {\bf I}_{k,n}}\rightarrow \KK[x_{i,j}]  \quad{\rm with}\quad P_I \mapsto \det(X_I)
\end{eqnarray}
where $X_I$ is the submatrix of $X = (x_{i,j})$ with columns indexed by $I$. The kernel of the map $\varphi_n$ is called the \textit{Pl\"ucker ideal of $\Gr(k,n)$} and we denote it by $G_{k,n}$.

\vspace{-2mm}
\subsection{Schubert and opposite Schubert varieties. }\label{sec:schubert_def}
Schubert and opposite Schubert varieties~are families of closed subvarieties of the Grassmannian $\Gr(k,n)$ which are indexed by ${\bf I}_{k,n}$. Note that there are a few equivalent ways to define Schubert varieties, see e.g.~\cite{lakshmibai2015grassmannian}. We consider the classical construction for these varieties as follows. Fix a basis $\{e_1, \dots, e_n \} \subseteq \KK^n$ and for each $i \in [k]$ define the subspaces $W_i = \langle e_1, \dots, e_i \rangle$ and
$W^i = \langle e_n, \dots, e_{n-i+1} \rangle$. For sets $v = \{v_1 < \dots < v_k\}$ and $ w = \{w_1 < \dots < w_k \} \in {\bf I}_{k,n}$ we define the corresponding Schubert variety $X_w$ and opposite Schubert variety $X^v$ to be
\[
X_w = \left\{V \in \Gr(k,n) : \dim(V \cap W_{w_j}) \ge j \textrm{ for all } j \in [k] \right\},\quad{\rm and}
\]
\[
X^v = \left\{ V \in \Gr(k,n) : \dim(V \cap W^{n-v_{(k-j+1)}+1}) \ge j \textrm{ for all } j \in [k]\right\}.
\]
Note that the sets $X_w$ and $X^v$ are closed subvarieties of $\Gr(k,n)$. 
There is a natural partial order on ${\bf I}_{k,n}$ given by $\{v_1 < \dots < v_k \} \le \{w_1 < \dots < w_k \}$ if and only if $v_1 \le w_1, \dots, v_k \le w_k$. This partial order allows us to see a number of properties of Schubert varieties.
\begin{remark}[Remark 5.3.4 and Corollary 5.3.5 \cite{lakshmibai2015grassmannian}]
Fix $v, w \in {\bf I}_{k,n}$ with $v = \{v_1 < \dots < v_k \}$. Let $X_v$ and $X_w$ be Schubert subvarieties of $\Gr(k,n)$.  Then the following hold:
\begin{enumerate}
    \item $\langle e_{v_1}, \dots, e_{v_k} \rangle \in X_w$ if and only if $v \le w$.
    \item $X_v \subseteq X_w$ if and only if $v \le w$.
    \item For all $I \in {\bf I}_{k,n}$, the function $P_I|_{X_w}$ is non-zero if and only if $I \le w$.
\end{enumerate}
Moreover, the Schubert variety $X_w \subseteq \PP^{\binom n k - 1}$ is precisely the zero set of the polynomials in the ideal $G_{k,n} +\langle P_I : I \nleq w\rangle$. 
It is often convenient to think of the ideal of the Schubert variety as a subset of the ring $\KK[P_I]_{I \le w}$.
\end{remark}

Analogous statements hold for the opposite Schubert variety $X^v$ whose ideal
is given by $G_{k,n} \cup \langle P_I : I \ngeq v \rangle$. Similarly, 
we think of the ideal of $X^v$ as a subset of the ring $\KK[P_I]_{I \ge v}$.

\vspace{-2mm}\subsection{Richardson varieties.}\label{intro:Richardson}

Fix $k \le n$ positive integers and let $v, w\in {\bf I}_{k,n}$. The Richardson variety $X_w^v$ associated to $v,w$ 
is defined as $X_w\cap X^v$.  We recall that $X_w^v\neq \emptyset$ if and only if $v\leq w$, see \cite[Corollary 2.1.2]{kreiman2002richardson}.
To fix our notation we define the set
\[
T^v_{w}=\{I \in {\bf I}_{k,n} : v \le I \le w\}.
\]
The associated ideal of $X_w^v$ is the restriction of $G_{k,n}$ to the ring $K[P_I]_{I \in T_w^v}$ or equivalently it is the sum of the ideals of the corresponding Schubert and opposite Schubert varieties:
\begin{align*}
I(X_w^v) &= G_{k,n}|_{T^v_{w}} := (G_{k,n} + \langle P_I: I\in {\bf I}_{k,n} \backslash T^v_{w}\rangle) \cap \mathbb K[P_I]_{I \in T^v_{w}} \\
&= (I(X_w)+I(X^v)) \cap \mathbb K[P_I]_{I \in T^v_{w}}.
\end{align*}
Schubert and opposite Schubert varieties are special examples of Richardson varieties, namely
$X_w = X_w^{id}$ and 
$X^v = X_{w_0}^v$ where $w_0 = \{n-k+1, \dots, n-1, n\}$ is the largest element of ${\bf I}_{k,n}$.

\begin{remark}
Richardson varieties are defined more generally for quotients of semi-simple algebraic groups $G$ by parabolic subgroups $P \subseteq G$ in~\cite{lakshmibai2003richardson}. When $G$ is of type $A_n$, they are indexed by permutations $v,w \in S_n$. In the classical formulation above for the Grassmannian, the parabolic subgroups are maximal. In this case, the permutations $v, w$ giving rise to distinct Richardson varieties can be taken to be pairs of \textit{Grassmann permutations}, which are left coset representatives of $S_n / (S_k \times S_{n-k})$. These coset representatives can be identified with subsets $I \in {\bf I}_{k,n}$. We can think of $I$ as a permutation which sends $[k]$ to $I$.
\end{remark}

\vspace{-2mm}
\subsection{Standard monomial basis.}
\label{rem:monomial_basis}
The Pl\"ucker algebra of the Grassmannian $\Gr(k,n)$ is the quotient of the polynomial ring $\KK[P_I]_{I\in{\bf I}_{k,n}}$ by the Pl\"ucker ideal $G_{k,n}$. 
In \cite{hodge1943some}, Hodge provided a combinatorial rule to choose a monomial basis for the Pl\"ucker algebra in terms of \textit{semi-standard Young tableaux}. A monomial $P = P_{I_1}P_{I_2} \dots P_{I_d}$ is called \textit{standard} if $I_1 \le \dots \le I_d$. The monomial $P$ is called \textit{standard for $X_w^v$} if $P$ is standard and $v \le I_i\le w$ for all $i \in [d]$.
It is convenient to write monomials as rectangular tableaux whose columns correspond to factors of the monomial. For example, the corresponding tableau of the monomial $P_I P_J$ with $I = \{i_1 < \dots < i_k\}$ and $J = \{j_1 < \dots < j_k\}$ is the following:
\begin{center}
    \begin{tabular}{|c|c|}
        \multicolumn{1}{c}{$I$} & 
        \multicolumn{1}{c}{$J$} \\
        \hline
        $i_1$ & $j_1$ \\
        \hline
        $i_2$ & $j_2$ \\
        \hline
        $\vdots$ & $\vdots$ \\
        \hline
        $i_k$ & $j_k$ \\
        \hline
    \end{tabular}
\end{center}

The tableaux corresponding to standard monomials are known as \textit{semi-standard Young tableaux}. These tableaux are defined by the property that the entries in each column are strictly increasing and the entries in each row are weakly increasing. 
 We recall that:

\begin{theorem}[Theorem 3.3.2 \cite{kreiman2002richardson}]\label{thm:SMT_Grassmannian}
The standard monomials for $X_w^v$ give a monomial basis for the associated algebra 
of the Richardson variety $\KK[P_I]_{I \in T^v_{w}} / I(X_w^v)$.
\end{theorem}

\subsection{Matching fields.}\hspace{0.3mm}\label{sec:prelim_matching_field}
A matching field is a combinatorial object that encodes a weight vector for the polynomial ring $\KK[P_J]_{I \in {\bf I}_{k,n}}$ which is \textit{induced} from a weight vector for the polynomial ring $\KK[x_{i,j}]$. Here, we recall \textit{block diagonal} matching fields from \cite{KristinFatemeh, OllieFatemeh}.

\begin{definition}\label{example:block}
Given integers $k,n$ and 
$0\leq \ell\leq n$, we fix the $k \times n$ matrix $M_\ell$ with entries:
\begin{eqnarray}\label{eq:matrix}
M_\ell(i,j)=\begin{cases} 
      (i-1)(n-j+1) & \textrm{if } i\neq 2, \\
      \ell-j+1 & \textrm{if } i=2, 1\leq j\leq \ell,\\
      n-j+\ell+1 & \textrm{if } i=2, \ell< j\leq n.\\
   \end{cases}
\end{eqnarray}
Recall that $X=(x_{i,j})$ is a $k \times n$ matrix of indeterminates. For each $k$-subset $J$ of $[n]$, the initial term of the Pl\"ucker form $\varphi_n(P_J) \in \mathbb{K}[x_{i,j}]$ denoted by $\ini_{M_\ell}(P_J)$ is the sum of all terms in $\varphi_n(P_J)$ of the lowest weight, where the weight of a monomial $\bf m$ is the sum of entries in $M_\ell$ corresponding to the variables in $\bf m$. We write $M_\ell(\textbf{m})$ for the weight of $\textbf{m}$. We prove below that $\ini_{M_\ell}(P_J)$ is a monomial for each subset $J \subseteq [n]$.
The weight of each variable $P_J$ is defined as the weight of each term of $\textrm{in}_{M_\ell}(P_J)$ with respect to $M_\ell$, and it is called {\em the weight induced by $M_\ell$}. We write $\wb_\ell$ for the weight vector induced by $M_\ell$. 
\end{definition}

\begin{lemma}\label{lem:wt_add_const}
Let $M = (m_{i,j})$ and $M' = (m'_{i,j})$ be $k \times n$ weight matrices. Suppose there exists $p \in \{ 1, \dots, n\}$ such that $m_{i,j} = m'_{i,j}$ for all $i \in [k]$ and $j \in [n] \backslash p$. Suppose that there exists $c \in \RR$ such that $m'_{i,p} = m_{i,p} + c$ for all $i \in \{1, \dots, k \}$. Then the initial terms of the Pl\"ucker forms are the same with respect to $M$ and $M'$.
\end{lemma}

\begin{proof}
Let $J$ be a $k$-subset of $[n]$. If $J$ does not contain $p$ then the submatrices of $M$ and $M'$ with columns indexed by $J$ coincide, hence the initial terms of the Pl\"ucker form $\varphi_n(P_J)$ with respect to $M$ and $M'$ are the same. On the other hand, if $J$ contains $p$ then consider each monomial $\textbf{x}$ in the Pl\"ucker form $\varphi_n(P_J)$. The monomial is squarefree and contains a unique variable of the form $x_{i,p}$ for some $i \in \{1, \dots, k \}$. Therefore, $M'(\textbf{x}) = M(\textbf{x}) + c$. Therefore, the initial term of $\varphi_n(P_J)$ is the same with respect to $M$ and $M'$.
\end{proof}

By the same method, one can prove an analogous result for weight matrices which differ by a constant in a particular row.

\begin{proposition}\label{prop:unique}
Let $M = M_0$ be the $k \times n$ weight matrix and $J$ 
be a $k$-subset. Then the initial term ${\rm in}_M(P_J)$ 
is the leading diagonal term, in particular, it is a monomial.
\end{proposition}

\begin{proof}
We show that the leading diagonal term of the Pl\"ucker form $\varphi_n(P_J)$ i.e.~$x_{1,j_1}x_{2,j_2}\cdots x_{k,j_k}$ where $j_1<j_2<\cdots<j_k$ equals to $\textrm{in}_M(P_J)$. We proceed by induction on $k$. For $k = 1$ the result holds trivially. So assume $k > 1$.~We~have
\[
\varphi_n(P_J) = \sum_{\sigma \in S_k} x_{1,j_{\sigma(1)}} \cdots x_{k,j_{\sigma(k)}}.
\]
For each $\sigma \in S_k$ such that $x_{1,j_{\sigma(1)}} \cdots x_{k,j_{\sigma(k)}}$ has minimum weight with respect to $M$, consider the value $\sigma(k) \in [k]$. 
Suppose $\sigma(k) = p$ for some $p \in [k]$. Then, by induction, we have that the leading term of the $\varphi_n(P_{J \backslash j_p})$ is the leading diagonal term. So $\sigma(1) = 1, \dots, \sigma(p-1) = {p-1}, \sigma(p) = {p+1}, \dots, \sigma(k-1) = k$ and $\sigma(k) = p$, therefore, the weight of the monomial is
\begin{align*}
M(x_{1,j_{\sigma(1)}} \cdots x_{k,j_{\sigma(k)}}) &= 
\sum_{i = 1}^{p-1} (i-1)(n - j_i +1) +
\sum_{i = p+1}^k (i-2)(n - j_i + 1)
+ (k - 1)(n - j_p + 1) \\
&= \sum_{i = 1}^{k} (i-1)(n - j_i +1) - \sum_{i = p+1}^k (n-j_i + 1) + (k-p)(n - j_p +1) \\
&= M(x_{1,j_1} \cdots x_{k,j_k}) + \sum_{i = p+1}^k j_i - (k - p)j_p.
\end{align*}
Note that $j_p < j_{p+1} < \dots < j_k$. So $\sum_{i = p+1}^k j_i - (k - p)j_p > 0$. If $\sigma(k) < k$ then the weight $M(x_{1,j_{\sigma(1)}} \cdots x_{k,j_{\sigma(k)}})$ is not minimum. So $\sigma(k) = k$ and we are done by induction.
\end{proof}

\vspace{-3mm}
\begin{proposition}\label{prop:unique_ell}
Let $\ell \in \{1, \dots, n-1 \}$, $k \ge 2$, $M = M_\ell$, the $k \times n$ weight matrix, and $J = \{j_1 < \dots < j_k \} \subset [n]$. 
Then the initial term of the Pl\"ucker form $\varphi_n(P_J)$ is given by
\[
{\rm in}_M(P_J) = \left\{
\begin{tabular}{@{}ll}
     $x_{1, j_1}x_{2, j_2}x_{3, j_3}\dots x_{k, j_k}$ & if $j_1 > \ell$ or $j_2 \le \ell$, \\
     $x_{1, j_2}x_{2, j_1}x_{3, j_3} \dots x_{k, j_k}$ & otherwise. 
\end{tabular}
\right.
\]
In particular, the leading term $\textrm{in}_M(P_J)$ is a monomial.
\end{proposition}

\begin{proof}
Suppose that $j_1 > \ell$. By definition, the weight matrices $M_\ell$ and $M_0$ differ only in the second row. The entries of the second row are
\begin{align*}
(M_0):      & \quad [n \ n-1 \ \dots \ 1], \\
(M_\ell) :  & \quad [\ell \ \ell-1 \ \dots \ 1 \ n \ n-1 \ \dots \ \ell+1].
\end{align*}
Consider the submatrices of $M_0$ and $M_\ell$ consisting of the columns indexed by $J$. Since $j_1 > \ell$ the second row entries differ by exactly $\ell$ in each respective position. And so by the row-version of Lemma~\ref{lem:wt_add_const}, the leading term of the Pl\"ucker form $\varphi_n(P_J)$ is the same with respect to $M_0$ and $M_\ell$. By Proposition~\ref{prop:unique}, the initial term is $\textrm{in}_M(P_J) = x_{1, j_1}x_{2, j_2}x_{3, j_3}\dots x_{k, j_k}$.

Suppose that $j_1 \le \ell$. 
We will prove the result by induction on $k$. For the case $k = 2$,
\[
M_\ell = 
\begin{bmatrix}
0    & 0      & \dots & 0 & 0 & 0   & \dots & 0      \\
\ell & \ell-1 & \dots & 1 & n & n-1 & \dots & \ell+1 
\end{bmatrix}.
\]
If $j_1 > \ell$ or $j_2 \le \ell$ then the leading term of the Pl\"ucker form $\varphi_n(P_J)$ is the leading diagonal term, i.e.~$\textrm{in}_M(P_J) = x_{1, j_1} x_{2, j_2}$. Otherwise we have $j_1 \le \ell$ and $j_2 > \ell$, and so the leading term of the Pl\"ucker form is the antidiagonal term, i.e.~$\textrm{in}_M(P_J) = x_{1, j_2} x_{2, j_1}$.

Suppose $k > 2$. For each $\sigma \in S_k$ such that $x_{1, j_{\sigma(1)}} \dots x_{k, j_{\sigma(k)}}$ has minimum weight with respect to $M_\ell$, consider the value $p = \sigma(k) \in [k]$. Let $J' = J \backslash j_p = \{j_1' < j_2' < \dots < j_{k-1}' \}$. There are two cases for $J'$, either $j_2' \le \ell$ or $j_2' > \ell$.

\smallskip

\textbf{Case 1.} Assume $j_2' \le \ell$. By induction we have $\textrm{in}_M(P_{J'}) = x_{1, j_1'} x_{2, j_2'} \dots x_{k-1, j_{k-1}'}$. And so we have $\sigma(1) = 1, \dots, \sigma(p-1) = p-1, \sigma(p) = p+1, \dots, \sigma(k-1) = k, \sigma(k) = p$. Suppose by contradiction that $p \le k-1$, then we have
\vspace{-.4cm}
\[
    M_\ell(x_{1, j_{\sigma(1)}}  \dots x_{k, j_{\sigma(k)}}) - M_\ell(x_{1, j_1}  \dots x_{k, j_k}) 
    = \sum_{i = p}^k \left(M_\ell(x_{i, j_{\sigma(i)}}) - M_\ell(x_{i, j_i})\right)
    = \sum_{i = p}^k (i - 1)(j_i - j_{\sigma(i)})
\]
\vspace{-.4cm}
\[
    = \left(\sum_{i = p}^{k-1} (i-1)(j_i - j_{i+1})\right) + (k-1)(j_k - j_p)
    = \sum_{i = p+1}^{k} (j_i - j_p) > 0.
\]
But by assumption $x_{1, j_{\sigma(1)}} \dots x_{k, j_{\sigma(k)}}$ has minimum weight, a contradiction. And so we have $p = k$ hence $\textrm{in}_M(P_J) = x_{1, j_1} x_{2, j_2} \dots x_{k, j_k}$.

\smallskip

\textbf{Case 2.} Assume $j_2' > \ell$. Either we have $j_1' \le \ell$ or $j_1' > \ell$. In this case assume further that $j_1' \le \ell$, we will show that $j_1' > \ell$ is impossible in Case 3. By induction we have $\textrm{in}_M(P_{J'}) = x_{1, j_2'} x_{2, j_1'} x_{3, j_3'} \dots x_{k-1, j_{k-1}'}$. Assume by contradiction that $k \neq p$. We proceed by taking cases on $p$, either $p = 1$, $p = 2$ or $3 \le p \le k - 1$.

\textbf{Case 2.1} Assume $p = 1$. So we have $\sigma(1) = 3, \sigma(2) = 2, \sigma(3) = 4, \dots, \sigma(k-1) = k, \sigma(k) = 1$. Since $j_p < j_1' \le \ell$, we have
\[
    M_\ell(x_{1, j_{\sigma(1)}}  \dots x_{k, j_{\sigma(k)}}) - M_\ell(x_{1, j_1} \dots x_{k, j_k}) 
    = \left(\sum_{i = 4}^{k} (j_i - j_1)\right) + 2(j_3 - j_1) 
    > 0.
\]
But by assumption $x_{1, j_{\sigma(1)}} \dots x_{k, j_{\sigma(k)}}$ has minimum weight, a contradiction.

\textbf{Case 2.2.} Assume $p = 2$. So we have $\sigma(1) = 3, \sigma(2) = 1, \sigma(3) = 4, \dots, \sigma(k-1) = k, \sigma(k) = 2$. Since $j_p < j_1' \le \ell$, we have
\[
    M_\ell(x_{1, j_{\sigma(1)}}  \dots x_{k, j_{\sigma(k)}}) - M_\ell(x_{1, j_1}  \dots x_{k, j_k}) 
    = \left(\sum_{i = 4}^{k} (j_i - j_2)\right) + (j_3 - j_2) + (j_3 - j_1)> 0.
\]
But by assumption $x_{1, j_{\sigma(1)}} \dots x_{k, j_{\sigma(k)}}$ has minimum weight, a contradiction.

\textbf{Case 2.3.} Assume $3 \le p \le k-1$. And so we have $\sigma(1) = 2, \sigma(2) = 1, \sigma(3) = 3, \dots, \sigma(p-1) = p-1, \sigma(p) = p+1, \dots, \sigma(k-1) = k, \sigma(k) = p$. Therefore,
\[
    M_\ell(x_{1, j_{\sigma(1)}}  \dots x_{k, j_{\sigma(k)}}) - M_\ell(x_{1, j_2} x_{2, j_1} x_{3, j_3} \dots x_{k, j_k}) 
    = \sum_{i = p+1}^{k} (j_i - j_p) > 0.
\]
But by assumption $x_{1, j_{\sigma(1)}} \dots x_{k, j_{\sigma(k)}}$ has minimum weight, a contradiction.

\textbf{Case 3.} Assume $j_1', j_2' > \ell$. By induction, we have $\textrm{in}_M(P_{J'}) = x_{1, j_1'} x_{2, j_2'} \dots x_{k-1, j_{k-1}'}$. Since $j_1 \le \ell$ we must have $j_1' = j_2, \dots j_{k-1}' = j_k $ and so $\sigma(1) = 2$, $\sigma(2) = 3,\dots,$  $\sigma(k-1) = k$ and $\sigma(k) = 1$. Therefore, 

$
    M_\ell(x_{1, j_{\sigma(1)}}  \dots x_{k, j_{\sigma(k)}}) 
    - M_\ell(x_{1, j_2} x_{2, j_1} x_{3, j_3} \dots x_{k, j_k})= 
    \left(\sum_{i = 3}^{k} (j_i - j_1)\right)  + j_3 + n > 0.
$

\noindent But by assumption $x_{1, j_{\sigma(1)}} \dots x_{k, j_{\sigma(k)}}$ has minimum weight, a contradiction. 
\end{proof}

\vspace{-2mm}

\begin{definition}\label{def:matching_field_perm}
Given integers $k$, $n$ and $0\leq\ell\leq n$, $M_\ell$ leads to a permutation for each subset $J= \{i_1, \dots, i_k\} \subset [n]$. More precisely, we think of $J$ as being identified with the Pl\"ucker form $\varphi_n(P_J)$ and
we consider the set to be ordered by $J = \{i_1 < \dots < i_{k}\}$. Since $\ini_{M_\ell}(P_J)$ is a unique term in the corresponding minor of $X=(x_{i,j})$, we have $\ini_{M_\ell}(P_J)=x_{1,i_{\sigma(1)}}\cdots x_{k,i_{\sigma(k)}}$ for some $\sigma \in S_k$, which we call the permutation associated to $M_\ell$. We represent the variable $\ini_{M_\ell}(P_J)$ as a $k \times 1$ tableau where the entry of $(j, 1)$ is $i_{\sigma(j)}$ for each $j \in [k]$. We can think of $M_\ell$ as inducing a new ordering on the elements of $J$ which can be read from the tableau.
\end{definition}
\vspace{-2mm}
\begin{remark}\label{def:block_weight}
By Propositions~\ref{prop:unique} and \ref{prop:unique_ell} the initial term $\ini_{M_\ell}(P_J)$ is a monomial for each Pl\"ucker form $\varphi_n(P_J)$ where $J = \{j_1 < \dots < j_k \} \subset [n]$. These propositions give a precise description of the initial terms and the induced weight on the Pl\"ucker variable $P_J$ as follows.
\[
{\bf w}_\ell(P_J) = \left\{
\begin{tabular}{@{}ll}
    $0$ & if $k = 1$, \\
    $(n+\ell + 1 - j_2) + \sum_{i = 3}^k (i - 1)(n+1-j_i)$ & if $k \ge 1$ and $|J \cap \{1, \ldots, \ell \}| = 0 $, \\
    $(\ell + 1 - j_1) + \sum_{i = 3}^k (i - 1)(n+1-j_i)$ & if $k \ge 1$ and $|J \cap \{1, \ldots, \ell \}| = 1$, \\
    $(\ell+1 - j_2) + \sum_{i = 3}^k (i - 1)(n+1-j_i)$ & if $|J \cap \{1, \ldots, \ell \}| \ge 2$.
\end{tabular}\right.
\]
\end{remark}
\begin{Notation}\label{not:ini}
For each $\alpha=(\alpha_J)_{J}$ in $\mathbb{Z}_{\geq 0}^{{\binom n k}}$ we fix the notation ${\bf P}^{{\bf \alpha}}$ denoting the monomial $\prod_{J}P_J^{\alpha_J}$. We denote ${\rm in}_{{\bf w}_\ell}(G_{k,n})$ for the initial ideal of $G_{k,n}$ with respect to $\wb_\ell$. 
This is defined as the ideal generated by polynomials $\inwb(f)$ for all polynomials $f \in G_{k,n}$, where 
\vspace{-1mm}
\[\inwb(f)=\sum_{\alpha_j\cdot \wb_\ell=d}{c_{{\bf \alpha}_j}\bf P}^{{\bf \alpha}_j}\quad\text{for}\quad f=\sum_{i=1}^t c_{{\bf \alpha}_i}{\bf P}^{{\bf \alpha}_i}\quad\text{and}\quad d=\min\{\alpha_i\cdot \wb_\ell:\ i=1,\ldots,t\}.\] 
\end{Notation}
\begin{remark}\label{def:block}
Propositions~\ref{prop:unique} and \ref{prop:unique_ell} show that the permutation given by $M_\ell$ and associated to $J$, which defines the matching field, is given by:
\[
 \BLambda(J)= \left\{
     \begin{array}{@{}ll}
      id  & \text{if $k=1$ or $\lvert J \cap \{1,\ldots,\ell\}\rvert \neq 1$},\\
      (12)  & \text{otherwise}, \\
     \end{array}
   \right.
\]
where $(12)$ is the transposition interchanging $1$ and $2$. The functions $B_{\ell}$ are called $2$-block diagonal matching fields in \cite{KristinFatemeh}. Note that $\ell=0$ or $n$ gives rise to the choice of the diagonal terms in each submatrix as in Example~\ref{ex:diag}. Such matching fields are called {\em diagonal}. See, e.g.~\cite[Example 1.3]{sturmfels1993maximal}.  Given a block diagonal matching field $\BLambda$ we define $B_{\ell,1} = \{1, \dots, \ell \}$ and $B_{\ell,2} = \{\ell + 1, \dots, n \}$.
\end{remark}
\vspace{-2mm}
\begin{example}\label{ex:diag}
Let $k=3$, $n=5$ and $\ell=0$, so the matching field $B_\ell$ is the diagonal matching field, with $B_{\ell,1}=\emptyset$ and $B_{\ell,2}=\{1,2,3,4,5\}$. We have
\[
M_0=
\begin{bmatrix}
     0  & 0 & 0  & 0  & 0 \\
     5  & 4 & 3  & 2  & 1 \\
     10 & 8 & 6  & 4  & 2 \\
\end{bmatrix} \textrm{ a weight matrix for }
X =
\begin{bmatrix}
     x_{11}  & x_{12} & x_{13}  & x_{14}  & x_{15} \\
     x_{21}  & x_{22} & x_{23}  & x_{24}  & x_{25} \\
     x_{31}  & x_{32} & x_{33}  & x_{34}  & x_{35} \\
\end{bmatrix}\ .
\]
The corresponding weight vector on $P_{123}, P_{124},\ldots,P_{345}$ is
${\bf w}_0 =(10, 8, 6, 7, 5, 4, 7, 5, 4, 4).$
Thus, for each $J=\{i < j < k\} \subseteq [5]$ we have that $\text{in}_{M_0} (P_J) = x_{1i}x_{2j}x_{3k}$. Therefore, the corresponding tableaux for $P_J$ are:
\[\begin{tabular}{|c|}
    \hline 1 \\ \hline 2 \\ \hline 3  \\ \hline
\end{tabular}\ ,\ 
\begin{tabular}{|c|}
    \hline
    1    \\
    \hline
    2  \\
    \hline
    4  \\
    \hline
\end{tabular}
\ ,\ \begin{tabular}{|c|}
    \hline
    1    \\
    \hline
    2  \\
    \hline
    5  \\
    \hline
\end{tabular}
\ ,\ \begin{tabular}{|c|}
    \hline
    1    \\
    \hline
    3  \\
    \hline
    4  \\
    \hline
\end{tabular}
\ ,\ \begin{tabular}{|c|}
    \hline
    1    \\
    \hline
    3  \\
    \hline
    5  \\
    \hline
\end{tabular}
\ ,\ \begin{tabular}{|c|}
    \hline
    1    \\
    \hline
    4  \\
    \hline
    5  \\
    \hline
\end{tabular}
\ ,\ \begin{tabular}{|c|}
    \hline
    2    \\
    \hline
    3  \\
    \hline
    4  \\
    \hline
\end{tabular}
\ ,\ \begin{tabular}{|c|}
    \hline
    2    \\
    \hline
    3  \\
    \hline
    5  \\
    \hline
\end{tabular}
\ ,\ \begin{tabular}{|c|}
    \hline
    2    \\
    \hline
    4  \\
    \hline
    5  \\
    \hline
\end{tabular}
\ ,\ \begin{tabular}{|c|}
    \hline
    3    \\
    \hline
    4  \\
    \hline
    5  \\
    \hline
\end{tabular}\ .
\]
\end{example}
\noindent Note that each initial term $\text{in}_{M_0}(P_J)$ is the leading diagonal term of the Pl\"ucker form $\varphi_n(P_J)$. Let us consider a block diagonal matching field which is not diagonal.

\begin{example}\label{ex:ell_3_mf}
Let $k=3$, $n=5$ and $\ell=3$. Then $B_{\ell,1}=\{1,2,3\}$, $B_{\ell,2}=\{4,5\}$ and
\[
M_2=\begin{bmatrix}
     0  & 0 & 0  & 0  & 0 \\
     3  & 2 & 1  & 5  & 4 \\
     10 & 8 & 6  & 4  & 2 \\
\end{bmatrix}.
\]
Comparing this matrix with $M_0$, the weight matrix for the diagonal case, we see that the only differences are in the second row. The entries of the second row are obtained by permuting the entries in the second row of $M_0$. The corresponding weight vector on the Pl\"ucker variables $P_{123}, P_{124},\ldots,P_{345}$ is
${\bf w}_2=(8, 6, 4, 5, 3, 5, 5, 3, 4, 3).$
For each $J=\{i,j,k\}$ we have the corresponding tableaux for $P_J$ which are
\[
\begin{tabular}{|c|}
    \hline
    1    \\
    \hline
    2  \\
    \hline
    3  \\
    \hline
\end{tabular}\ ,\ \begin{tabular}{|c|}
    \hline
    1    \\
    \hline
    2  \\
    \hline
    4  \\
    \hline
\end{tabular}
\ ,\ \begin{tabular}{|c|}
    \hline
    1    \\
    \hline
    2  \\
    \hline
    5  \\
    \hline
\end{tabular}
\ ,\ \begin{tabular}{|c|}
    \hline
    1    \\
    \hline
    3  \\
    \hline
    4  \\
    \hline
\end{tabular}
\ ,\ \begin{tabular}{|c|}
    \hline
    1    \\
    \hline
    3  \\
    \hline
    5  \\
    \hline
\end{tabular}
\ ,\ \begin{tabular}{|c|}
    \hline
    4    \\
    \hline
    1  \\
    \hline
    5  \\
    \hline
\end{tabular}
\ ,\ \begin{tabular}{|c|}
    \hline
    2    \\
    \hline
    3  \\
    \hline
    4  \\
    \hline
\end{tabular}
\ ,\ \begin{tabular}{|c|}
    \hline
    2    \\
    \hline
    3  \\
    \hline
    5  \\
    \hline
\end{tabular}
\ ,\ \begin{tabular}{|c|}
    \hline
    4    \\
    \hline
    2  \\
    \hline
    5  \\
    \hline
\end{tabular}
\ ,\ \begin{tabular}{|c|}
    \hline
    4    \\
    \hline
    3  \\
    \hline
    5  \\
    \hline
\end{tabular}\ .
\]
\end{example}

\subsection{Matching field ideals.}\label{sec:ideals}
A matching field also admits the data of a monomial map $\mathbb K[P_J] \rightarrow \mathbb K[x_{i,j}]$ which takes each $P_J$ to a term of the Pl\"ucker form $\varphi_n (P_J) = \det(X_J) \in \mathbb K[x_{i,j}]$. We define the matching field ideal of $B_\ell$ 
to be the kernel of the monomial map 
\begin{eqnarray}\label{eqn:monomialmap}
\phi_{\ell} \colon\  \mathbb{K}[P_I]_{I \in {\bf I}_{k,n}}  \rightarrow \mathbb{K}[x_{i,j}]  
\quad\text{with}\quad
 P_{J}   \mapsto 
 \ini_{M_\ell}(P_J)
\end{eqnarray}
where $M_\ell$ is the matrix in \eqref{eq:matrix}.
We will show in Theorem~\ref{thm:main_intro} that whenever a $B_\ell$ gives rise to a toric initial ideal, then the corresponding toric ideal is equal to the matching field ideal. 

\begin{Notation}
\label{def:restricted}
Fix a Richardson variety $X_w^v$.
We write $\phi_\ell|_w^v$ for the restriction of the map \eqref{eqn:monomialmap} to the variables $P_J$ with $J\in T_{w}^v$ and $\ker(\phi_\ell|_w^v)$ for the associated matching field ideal. 
\end{Notation}

\vspace{-3mm}

\section{Monomial-free ideals arising from matching fields}\label{sec:gr}

Here,  we define the ideals $G_{k,n,\ell}|_w^v$ and characterise when they are monomial-free, see Theorem~\ref{thm:Rich}. Our main results in this paper 
precisely apply when these ideals are monomial-free.
We first summarise some of the important properties of  matching field ideals 
from \cite{OllieFatemeh}. 
\begin{theorem}[Theorems 4.1, 4.3 and Corollary 4.7 \cite{OllieFatemeh}]\label{thm:JAlebra}
  The initial ideal ${\rm in}_{{\bf w}_\ell}(G_{k,n})$ is toric and it is generated by quadratic binomials. Moreover, for the matching field ideal $\ker(\phi_\ell)$,
  \begin{equation}\label{eq:JAlgebra}
      {\rm in}_{{\bf w}_\ell}(G_{k,n})=\ker(\phi_\ell).
  \end{equation}
\end{theorem}
The following notation will simplify the description of the ideals throughout this note.

\begin{Notation}
\label{notation:restricted}
Let $G \subset \mathbb{K}[P_I]_{I \in {\bf I}_{k,n}}$ be a collection of polynomials and $T$ be a collection of subsets of $[n]$. We identify $T$ with the characteristic vector of $T^{\rm c}$ that is $T_J = 1$ if $J \not\in T$ otherwise $T_J = 0$. For each $g \in G$ we write $g = \sum_{\alpha} c_\alpha {\bf P}^\alpha$ and define
\[
\hat g = \sum_{T \cdot\alpha = 0} c_\alpha {\bf P}^\alpha\quad\text{and}\quad 
G|_T= \{\hat g : g \in G \} \subseteq \mathbb{K}[P_I]_{I\in T}.
\]
\end{Notation}
We call $\langle G|_T\rangle $ the {\it restriction} of the ideal $\langle G\rangle$ to $T$. With the notation above we have:
\begin{lemma}[Lemma~6.3 in \cite{OllieFatemeh3}]\label{lem:elim_ideal_gen_set}
$\langle G|_T \rangle = \langle G \cup \{P_J:\ J\not\in T\} \rangle \cap \mathbb{K}[P_I]_{I \in T}$.
\end{lemma}

It is useful to think of $G|_T$ as the set obtained from $G$ by setting the variables $\{P_J: J \not\in T\}$ to zero. We say that the variable $P_J$ \textit{vanishes} in the ideal $\langle G|_T \rangle$ if $J \not\in T$. Similarly, we say that a polynomial $g$ vanishes in the restricted ideal $\langle G|_T \rangle$ if $g \in \langle P_J: J \not\in T \rangle$. 
The ideal $\langle G|_T \rangle$ can be computed in $\mathtt{Macaulay2}$~\cite{M2} as an elimination ideal using the command
$$\mathtt{eliminate}(\langle G\rangle + \langle P_J: J \not\in T \rangle, \{P_J : J \in T\}).$$

\begin{Notation}
\label{def:S}\label{def:matching_field_ideal_grassmannian}
Let $w=\{w_1,\ldots,w_k\}$ and $v=\{v_1,\ldots,v_k\} \in {\bf I}_{k,n}$ with $v\leq w$, and recall our notation from \S\ref{intro:Richardson}.
We denote the restricted ideals of ${\rm in}_{{\bf w}_\ell}(G_{k,n})$
as follows. 
\begin{equation}
\label{eq:gr}
    G_{k,n,\ell}|_w:={\rm in}_{{\bf w}_\ell}(G_{k,n})|_{T_{w}^{id}},
\ \ G_{k,n,\ell}|^v:={\rm in}_{{\bf w}_\ell}(G_{k,n})|_{T_{w_0}^v}\ \ {\rm and}\ \
G_{k,n,\ell}|_w^v:={\rm in}_{{\bf w}_\ell}(G_{k,n})|_{T_{w}^v}.
\end{equation}
Note that by Theorem~\ref{thm:JAlebra} the ideal ${\rm in}_{{\bf w}_\ell}(G_{k,n})$ is generated by a set of quadratic binomials whose restrictions to the set $T_{w}^v$ generate the above ideals by Lemma~\ref{lem:elim_ideal_gen_set}.
\end{Notation}

We are now ready to completely characterise 
monomial-free ideals of form $G_{k,n,\ell}|_w^v$.

\begin{theorem}\label{thm:Rich} Fix $k < n$ and $v,w \in {\bf I}_{k,n}$ with $v \le w$. 

\begin{itemize}
    \item If $\ell = 0$ or $\ell > n-k+1$ then the ideals $G_{k,n,\ell}|_w^v, G_{k,n,\ell}|_w$ and $G_{k,n,\ell}|^v$ are monomial-free. 
    \item Let $\ell \in \{1, \dots, n-k+1\}$, then the following hold:
\begin{itemize}
    \item[{\rm (i)}] The ideal $G_{k,n,\ell}|_w$ is monomial-free if and only if $w \in \MT_{k,n,\ell}$ which is the set of $\{w_1 < \dots < w_k\} \in {\bf I}_{k,n}$ such that at least one of the following hold:
    \begin{itemize}
        \item $w_1 \in\{1, \ell, n-k+1\}$,
        \item $w_2 \in \{1, \dots, \ell, w_1+1 \}$.
    \end{itemize}

    \item[{\rm (ii)}] The ideal $G_{k,n,\ell}|v$ is monomial-free if and only if $v \in \MT_{k,n,\ell}^{\rm opp}$ which is the set of $\{v_1 < \dots < v_k\} \in {\bf I}_{k,n}$ such that at least one of the following hold:
    \begin{itemize}
        \item $v_1 \in \{\ell+1, \dots, n\}$,
        \item $v_2 \in \{v_1+1, \ell+1\}$.
    \end{itemize}
    
    \item[{\rm (iii)}] The ideal $G_{k,n,\ell}|_w^v$ is monomial-free if and only if $w \in \MT_{k,n,\ell}$ and $v \in \MT_{k,n,\ell}^{\rm opp}$.
\end{itemize}
\end{itemize}\end{theorem}

\begin{proof}
Suppose that $\ell = 0$ or $\ell > n-k+1$. By \cite[Theorem~5.7]{OllieFatemeh} we have that $G_{k,n,\ell}|_w$ is monomial-free. By Lemma~\ref{lem:gr_zero_opposite} we have that $G_{k,n,\ell}|^v$ is monomial-free. The proof that $G_{k,n,\ell}|_w^v$ is monomial-free follows from part (iii) of this proof.

\smallskip

{\bf (i)} This part follows immediately from \cite[Theorem~5.7]{OllieFatemeh}.

\smallskip

{\bf (ii)} We will show that $G_{k,n,\ell}|^v$ contains a monomial if and only if both $v_1\in \{1, \dots, \ell \}$ and $v_2\in \{v_1+2, \ldots, n\}\setminus \{\ell+1\}$.

Take $v \notin \MT_{k,n,\ell}^{\rm opp}$. We begin by showing that $G_{k,n,\ell}|^v$ contains a monomial by taking cases on $v_2$. Note that $v_2 \neq \ell + 1$.

\textbf{Case 1.} Let $v_2 \le \ell$. Consider the following sets which we write in the true order according to the matching field. Let
\[
I = \{\ell + 1, v_2 - 1, n - k + 3, n-k+4, \dots, n - 1, n\}, \quad
J = \{v_1, v_2, n - k + 3, n-k+4, \dots, n - 1, n\},
\]
\[
I' = \{v_1, v_2 - 1, n - k + 3, n-k+4, \dots, n - 1, n\} \text{ and }
J' = \{\ell+1, v_2, n - k + 3, n-k+4, \dots, n - 1, n\}.
\]
By construction we have that $P_IP_J - P_{I'}P_{J'}$ is a binomial in $\init_{{\bf w}_\ell}(G_{k,n})$. We have that $I' \not\ge v$ hence $P_{I'}$ vanishes in $G_{k,n,\ell}|^v$. However, $I \ge v$ and $J \ge v$ hence $P_IP_J$ appears as a monomial in $G_{k,n,\ell}|^v$.

\textbf{Case 2.} Let $v_2 \ge \ell + 2$. We now prove that $v_2 + k - 1 \le n$. Suppose by contradiction that $v_2 + k - 1 > n$ then it follows that $v_3 = v_2 + 1$, $v_4 = v_3 + 1$, \dots, $v_k = n$. Now we have that $w_0v = (1,2, \dots, k-1, n - v_1 + 1) \in Z_{k,n}$. By Lemma~\ref{lem:gr_zero_opposite}, $G_{k,n,\ell}|^v$ is zero, a contradiction. Therefore, $v_2 + k - 1 \le n$. It follows that there exists $j \in \{2, \dots, k\}$ such that $v_j + 1 \le n$ and $v_j + 1 \not\in\{v_{j+1}, v_{j+2}, \dots, v_k\}$. Consider the following sets which we write in the true order according to the matching field. Let
\[ 
I = \{v_2, v_1, v_3, \dots, v_k \}, \quad
J = \{\ell + 1, v_2 + 1, v_3 + 1, \dots, v_{j-1} + 1, v_{j} + 1, v_{j+1}, v_{j+2}, \dots, v_k\},
\]
\[ 
I' = \{\ell+1, v_1, v_3, \dots, v_k \} \text{ and }
J' = \{v_2, v_2 + 1, v_3 + 1, \dots, v_{j-1} + 1, v_{j} + 1, v_{j+1}, v_{j+2}, \dots, v_k \}.
\]
By construction we have that $P_IP_J - P_{I'}P_{J'}$ is a binomial in $\init_{{\bf w}_\ell}(G_{k,n})$. Since $v_2 \ge \ell + 2$, we have that $I' < v$ hence $P_{I'}$ vanishes in $G_{k,n,\ell}|^v$. However, $I \ge v$ and $J \ge v$ hence $P_IP_J$ appears as a monomial in $G_{k,n,\ell}|^v$.

For the converse we assume that $G_{k,n,\ell}|^v$ contains a monomial. If $\ell > n - k +1$ or $\ell=0$ then by Lemma~\ref{lem:gr_zero_opposite} we have that $G_{k,n,\ell}|^v$ is monomial-free, a contradiction. So we may assume that $\ell \le n-k+1$. Suppose by contradiction that $v_1 \not\in B_{\ell,1}$ then $v_1 \ge \ell+1$. Suppose $P_I P_J$ is a monomial appearing in $G_{k,n,\ell}|^v$. In particular, $P_I$ and $P_J$ do not vanish so we have that $I, J \ge v$. We deduce that $I \cap B_{\ell,1} = \emptyset$ and $J \cap B_{\ell,1}= \emptyset$. Suppose that the monomial $P_IP_J$ is obtained from the binomial $P_IP_J - P_{I'}P_{J'}$ in $\inwb(G_{k,n})$. Then we have that $I' \cap B_{\ell,1} = \emptyset$ and $J' \cap B_{\ell,1} = \emptyset$. Therefore, the true ordering on all indices $I, I', J, J'$ is the diagonal order. It follows that the same monomial must appear in the ideal $G_{k,n,0}|^v$. However, by Lemma~\ref{lem:gr_zero_opposite}, $G_{k,n,0}|^v$ is monomial-free, a contradiction. So we may assume that $v_1 \in B_{\ell,1}$.
It remains to show that if $G_{k,n,\ell}|^v$ contains a monomial then $v_2 \in \{v_1 + 2, \dots, n \} \backslash \{\ell + 1 \}$. By the above argument, we may assume that $1\le \ell \le n - k +1$ and $v_1 \in B_{\ell,1}$. Assume by contradiction that $v_2 \not \in \{v_1 + 2, \dots, n \} \backslash \{\ell + 1\}$. Then there are two cases, either $v_2 = v_1 + 1$ or $v_2 = \ell + 1$.

\textbf{Case 1.} Let $v_2 = v_1+1$ and $P_I P_J$ be a monomial in $G_{k,n,\ell}|^v$ arising from a binomial $P_I P_J - P_{I'} P_{J'}$ in $\init_{{\bf w}_\ell}(G_{k,n})$. Let us write $I = \{i_1 < \dots < i_k \}$  and $J = \{j_1 < \dots < j_k \}$. By assumption we have $I, J \ge v$ so in particular, $i_1, j_1 \ge v_1$ and $i_2, j_2 \ge v_2$. It is easy to see that $B_{\ell}(I) \neq B_{\ell}(J)$ otherwise it follows that $P_{I'}P_{J'}$ does not vanish in $G_{k,n,\ell}|^v$.
So without loss of generality, assume that $B_{\ell}(I) = id$ and $B_{\ell}(J) = (12)$. So, in tableau form, the binomial $P_I P_J - P_{I'} P_{J'}$ is given by
\[
\begin{tabular}{|c|c|}
    \multicolumn{1}{c}{$I$} & \multicolumn{1}{c}{$J$} \\
    \hline
    $i_1$ & $j_2$  \\
    \hline
    $i_2$ & $j_1$  \\
    \hline
    $\vdots$ & $\vdots$ \\
    \hline
\end{tabular}
\, - \,
\begin{tabular}{|c|c|}
    \multicolumn{1}{c}{$I'$} & \multicolumn{1}{c}{$J'$} \\
    \hline
    $i_1$ & $j_2$  \\
    \hline
    $j_1$ & $i_2$  \\
    \hline
    $\vdots$ & $\vdots$ \\
    \hline
\end{tabular} \, .
\]
Note, we must have the first two rows of these two tableaux are different, otherwise $P_{I'}P_{J'}$ does not vanish in $G_{k,n,\ell}|^v$. By assumption we have $i_2, j_2 \ge v_2$ hence $P_{J'}$ does not vanish in $G_{k,n,\ell}|^v$. Hence $P_{I'}$ must vanish. We take cases on $B_{\ell}(I')$.

\textbf{Case 1.1.} Let $B_{\ell}(I') = id$. Then we have $i_1 < j_1$. Since $P_{I'}$ vanishes, we must have $j_1 < v_2 = v_1 + 1$. Therefore, $i_1 < v_1$, a contradiction. 

\textbf{Case 1.2.} Let $B_{\ell}(I') = (12)$. Then we have $j_1 < i_1$. 
Since $P_{I'}$ vanishes we must 
have $i_1 < v_2 = v_1 + 1$, and so $j_1 < v_1$ which is a contradiction.

\textbf{Case 2.} Let $v_2 = \ell +1$. Let $P_I P_J$ be a monomial in $G_{k,n,\ell}|^v$ arising from a binomial $P_I P_J - P_{I'} P_{J'}$ in $\inwb(G_{k,n})$ and write
$
I = \{i_1 < \dots < i_k \}$ and 
$J = \{j_1 < \dots < j_k \}.
$
By assumption we have $I, J \ge v$ so in particular, $i_1, j_1 \ge v_1$ and $i_2, j_2 \ge v_2$. It is easy to see that $B_{\ell}(I) \neq B_{\ell}(J)$ otherwise it follows that $P_{I'}P_{J'}$ does not vanish in $G_{k,n,\ell}|^v$. So without loss of generality, assume that $B_{\ell}(I) = id$ and $B_{\ell}(J) = (12)$. So, in tableau form, the binomial $P_I P_J - P_{I'} P_{J'}$ is given by
\[
\begin{tabular}{|c|c|}
    \multicolumn{1}{c}{$I$} & \multicolumn{1}{c}{$J$} \\
    \hline
    $i_1$ & $j_2$  \\
    \hline
    $i_2$ & $j_1$  \\
    \hline
    $\vdots$ & $\vdots$ \\
    \hline
\end{tabular}
\, - \,
\begin{tabular}{|c|c|}
    \multicolumn{1}{c}{$I'$} & \multicolumn{1}{c}{$J'$} \\
    \hline
    $i_1$ & $j_2$  \\
    \hline
    $j_1$ & $i_2$  \\
    \hline
    $\vdots$ & $\vdots$ \\
    \hline
\end{tabular} \, .
\]
Note that the first two rows of these tableaux must be different, otherwise $P_{I'}P_{J'}$ does not vanish in $G_{k,n,\ell}|^v$. By assumption we have $i_2, j_2 \ge v_2 = \ell + 1$ hence $P_{J'}$ does not vanish in $G_{k,n,\ell}|^v$. Hence $P_{I'}$ must vanish. Since $j_1 \in B_{\ell,1}$, we must have $i_1 < v_2$. Since $B_{\ell}(I) = id$ and $i_2 \ge v_2 = \ell + 1 \in B_{\ell,2}$, we must have $i_1 \in B_{\ell,2}$. So $i_1 \ge \ell + 1$, a contradiction.
So we have shown that $v_2 \in \{v_1 + 2, \dots, n \} \backslash \{\ell + 1 \}$. Thus $v$ satisfies all desired conditions.

\smallskip

{\bf (iii)} Given parts (i) and (ii), this part is equivalent to showing that $G_{k,n,\ell}|_w^v$ is monomial-free if and only if both $ G_{k,n,\ell}|_w$ and $G_{k,n,\ell}|^v$ are monomial-free. By definition we have 
\vspace{-2mm}
\begin{align*}
    G_{k,n,\ell}|_w^v &= (\init_{{\bf w}_\ell}(G_{k,n}) + \langle P_I : I \in {\bf I}_{k,n} \backslash T_{w}^v \rangle) \cap \mathbb K[P_I]_{I \in T_{w}^v} \\
    &= (\init_{{\bf w}_\ell}(G_{k,n}) + \langle P_I : I \nleq w \rangle + \langle P_I : I \ngeq v \rangle) \cap \mathbb K[P_I]_{I \in T_{w}^v} \\
    &= \left((\init_{{\bf w}_\ell}(G_{k,n}) + \langle P_I : I \nleq w \rangle) \cap \mathbb K[P_I]_{I \in T_{w}^v}\right) \\
    &\quad + \left((\init_{{\bf w}_\ell}(G_{k,n}) + \langle P_I : I \ngeq v \rangle) \cap \mathbb K[P_I]_{I \in T_{w}^v}\right)\\
    &= G_{k,n,\ell}|_w + G_{k,n,\ell}|^v \subseteq \mathbb K[P_I]_{I \in T_{w}^v}.
\end{align*}
In the above we consider $G_{k,n,\ell}|_w$ and $G_{k,n,\ell}|^v$ as ideals of the ring $\mathbb K[P_I]_{I \in T_{w}^v}$ by inclusion of their generators.
On the one hand if $G_{k,n,\ell}|^v$ or $G_{k,n,\ell}|_w$ contain a monomial then the same monomial appears in $G_{k,n,\ell}|_w^v$.
On the other hand suppose that $G_{k,n,\ell}|_w^v$ contains a monomial $P_I P_J$ then $v \le I, J \le w$. Also there exists $I', J' $ such that $P_I P_J - P_{I'} P_{J'}$ is a binomial in $\init_{{\bf w}_\ell}(G_{k,n})$ and either $I', J' \not\ge v$ or $I', J' \not\le w$. If $I', J' \not\ge v$ then $P_I P_J$ is a monomial in $G_{k,n,\ell}|^v$. If $I', J' \not\le w$ then $P_I P_J$ is a monomial in $G_{k,n,\ell}|_w$.
\end{proof}

The proofs above rely on Lemma~\ref{lem:gr_zero_opposite} and which follows from the following key but straightforward observation.

\begin{lemma}[Key Lemma] \label{lem:key_lemma}
Let $I, J \in {\bf I}_{k,n}$. Then $I \le J$ if and only if $w_0 I \ge w_0 J$.
\end{lemma}

The consequence of this observation is the following lemma which characterises the zero and monomial-free ideals for the opposite Schubert variety in the diagonal case, i.e. $\ell = 0$.

\begin{lemma} \label{lem:gr_zero_opposite}
We have the following:
\begin{itemize}
    \item[{\rm (i)}] The ideal $G_{k,n,\ell}|_w$ is zero if and only if $w \in Z_{k,n}$, where
    \[
Z_{k,n} = \{(1,2, \dots, k-1, i) : k \le i \le n \} \cup \{ (1, \dots, \hat{i}, \dots, k, k+1) : 1 \le i \le k-1 \}.
\]
    \item[{\rm (ii)}] The ideal $G_{k,n,0}|^v$ is zero if and only if $w_0v \in Z_{k,n}$,
    \item[{\rm (iii)}] If $\ell=0$ or $\ell>n-k+1$, then the ideal $G_{k,n,\ell}|^v$ is monomial-free.
\end{itemize}
\end{lemma}

\begin{proof}
{\bf (i)} This statement follows directly from \cite[Theorem 5.7]{OllieFatemeh}.

{\bf (ii)} For the first statement, note that $G_{k,n,0}|^v$ is non-zero if and only if there exists a binomial $P_I P_J - P_{I'} P_{J'}$ in $\init_{{\bf w}_0}(G_{k,n})$ such that $I, J \ge v$. By Lemma~\ref{lem:key_lemma}, $I, J \ge v $ if and only if $w_0I, w_0J \le w_0v$. Observe that $P_{w_0I} P_{w_0J} - P_{w_0I'} P_{w_0J'}$ is also a binomial in $\init_{{\bf w}_0}(G_{k,n})$ and all binomials can be written in this form since $w_0^2 = id$. Therefore, $G_{k,n,0}|^v$ is non-zero if and only if $G_{k,n,0}|_{w_0v}$ is non-zero.

{\bf (iii)} This statement is a consequence of part (i) and the bijection between binomials described above. If $G_{k,n,\ell}|^v$ contains a monomial $P_I P_J$ then there exists a binomial $P_I P_J - P_{I'} P_{J'}$ in $\init_{{\bf w}_0}(G_{k,n})$ such that $I, J \ge v$ and $I', J' < v$. By Lemma~\ref{lem:key_lemma}, $w_0I, w_0J \le w_0v$ and $w_0I', w_0J' > w_0v$. Since $P_{w_0I} P_{w_0J} - P_{w_0I'} P_{w_0J'}$ is also a binomial in $\init_{{\bf w}_0}(G_{k,n})$, therefore, $P_{w_0I} P_{w_0J}$ is a monomial in $G_{k,n,\ell}|_{w_0v}$, which contradicts part (i).
\end{proof}

\vspace{-5mm}
\section{Standard monomials for Richardson varieties}\label{sec:Standard_Gr}\label{sec:standard_monomial}

In this section, we 
provide a bijection between the semi-standard Young tableaux with two columns and the set of standard monomials for $\KK[P_I]_{I\in T_w^v}/\ker(\phi_\ell|_w^v)$ of degree two. 
\begin{definition}\label{def:Gamma_ell_deg2}
Let $T$ be a semi-standard Young tableau with two columns and $k$ rows whose entries lie in $[n]$, see \S\ref{rem:monomial_basis}. For each $\ell \in \{1, \dots, n-1\}$ we define the map $\Gamma_\ell : T \mapsto T'$ where $T'$ is a tableau whose columns are ordered according to the matching field $B_\ell$. Suppose that the entries of the columns of $T$ are $I = \{i_1 < i_2 < \dots < i_k \}$ and $J = \{j_1 < j_2 < \dots < j_k\}$. Since $T$ is in a semi-standard form, we assume that $i_s \le j_s$ for each $s \in [k]$. We define $T'$ as the tableau whose columns are $I'$ and $J'$ as sets and are ordered by the matching field $B_\ell$. The sets $I'$ and $J'$ are defined as follows.
\vspace{-2mm}
\begin{itemize}
    \item If $i_1, i_2, j_1 \in \{1, \dots, \ell \}$, $j_2 \in \{\ell+1, \dots, n \}$ and $i_1 < j_1 < i_2$ then we define $I' = \{j_1 < i_2 < i_3 < \dots < i_k \}$ and $J' = \{i_1 < j_2 < j_3 < \dots < j_k\}$.
    \item If $i_1 \in \{1, \dots, \ell \}$, $i_2, j_1, j_2 \in \{\ell+1, \dots, n \}$ and $j_1 < i_2 < j_2$ then we define $I' = \{j_1 < i_2 < i_3 < \dots < i_k \}$ and $J' = \{i_1 < j_2 < j_3 < \dots < j_k\}$
    \item Otherwise we define $I' = I$ and $J' = J$.
\end{itemize}
\end{definition}

\begin{lemma}\label{lem:SSYT_Gamma_injective}
Let $T_1$ and $T_2$ be semi-standard Young tableaux. If $\Gamma_\ell(T_1)$ and $\Gamma_\ell(T_2)$ are row-wise equal then $T_1$ and $T_2$ are equal.
\end{lemma}

\begin{proof}
We begin by noting that all rows except possibly the first two rows of a tableau are fixed by $\Gamma_\ell$. So it remains to show that if the first two rows of $\Gamma_\ell(T_1)$ and $\Gamma_\ell(T_2)$ are row-wise equal then so are the first two rows of $T_1$ and $T_2$. We also note that $\Gamma_\ell$ preserves the entries of a tableau, thought of as a multi-set. Let us assume by contradiction that $T_1$ and $T_2$ are not row-wise equal. By the above facts we may assume without loss of generality that
\[
T_1 = 
\begin{tabular}{|c|c|}
\hline
    $i_1$ & $j_1$  \\
\hline
    $i_2$ & $j_2$ \\
\hline
    \vdots & \vdots \\
\hline
\end{tabular} \, ,
\quad
T_2 = 
\begin{tabular}{|c|c|}
\hline
    $i_1$ & $i_2$  \\
\hline
    $j_1$ & $j_2$ \\
\hline
    \vdots & \vdots \\
\hline
\end{tabular}
\]
and $j_1 < i_2$. We proceed by taking cases on $s = |\{i_1, i_2, j_1, j_2 \} \cap \{1, \dots, \ell \}|$.

\textbf{Case 1.} Assume $s = 0$ or $4$. It follows that $\Gamma_\ell$ fixes $T_1$ and $T_2$. By row-wise equality of the second row of $\Gamma_\ell(T_1)$ and $\Gamma_\ell(T_2)$ we have that $j_1 = i_2$, a contradiction.

\textbf{Case 2.} Assume $s = 1$. It follows that $i_1 \in \{1, \dots, \ell \}$ and $i_2, j_1, j_2 \in \{\ell+1, \dots, n \}$. Since $j_1 < i_2$ we have
\[
\Gamma_\ell(T_1) = 
\begin{tabular}{|c|c|}
\hline
    $j_1$ & $j_2$  \\
\hline
    $i_2$ & $i_1$ \\
\hline
    \vdots & \vdots \\
\hline
\end{tabular} \, ,
\quad
\Gamma_\ell(T_2) = 
\begin{tabular}{|c|c|}
\hline
    $j_1$ & $i_2$  \\
\hline
    $i_1$ & $j_2$ \\
\hline
    \vdots & \vdots \\
\hline
\end{tabular}\, .
\]
By row-wise equality of the second row, we have that $j_2 = i_2$. However, in the tableau $T_2$, we have that $i_2 < j_2$, a contradiction.

\textbf{Case 3.} Assume $s = 2$. Since $j_1 < i_2$, it follows that $i_1, j_1 \in \{1, \dots, \ell \}$ and $i_2, j_2 \in \{\ell+1, \dots, n \}$. And so we have
\[
\Gamma_\ell(T_1) = 
\begin{tabular}{|c|c|}
\hline
    $i_2$ & $j_2$  \\
\hline
    $i_1$ & $j_1$ \\
\hline
    \vdots & \vdots \\
\hline
\end{tabular} \, ,
\quad
\Gamma_\ell(T_2) = 
\begin{tabular}{|c|c|}
\hline
    $i_1$ & $i_2$  \\
\hline
    $j_1$ & $j_2$ \\
\hline
    \vdots & \vdots \\
\hline
\end{tabular}\, .
\]
By row-wise equality of the second row the tableau we have that $i_1 = j_2$, a contradiction since $i_1 \in \{1, \dots, \ell \}$ and $j_2 \in \{\ell+1, \dots, n \}$.

\textbf{Case 4.} Assume that $s = 3$. It follows that $i_1, i_2, j_1 \in \{ 1, \dots, \ell\}$ and $j_2 \in \{\ell+1, \dots, n \}$. And so we have
\[
\Gamma_\ell(T_1) = 
\begin{tabular}{|c|c|}
\hline
    $j_1$ & $j_2$  \\
\hline
    $i_2$ & $i_1$ \\
\hline
    \vdots & \vdots \\
\hline
\end{tabular} \, ,
\quad
\Gamma_\ell(T_2) = 
\begin{tabular}{|c|c|}
\hline
    $i_1$ & $j_2$  \\
\hline
    $j_1$ & $i_2$ \\
\hline
    \vdots & \vdots \\
\hline
\end{tabular}\, .
\]
By row-wise equality of the second row, we have that $i_1 = j_1$. However, in $T_2$, we have that $i_1 < j_1$, a contradiction.
\end{proof}

\begin{lemma}\label{lem:SSYT_Gamma_surj_any_tableau}
Let $T$ be any tableau whose columns are valid for the block diagonal matching field $B_\ell$. Then there exists a semi-standard Young tableau $T'$ such that $\Gamma_\ell(T')$ and $T$ are row-wise equal.
\end{lemma}

\begin{proof}
Let $T$ be the tableau with entries $\{i_1, i_2 < i_3 < \dots < i_k \}$ and $\{j_1, j_2 < j_3 < \dots < j_k \}$,
\[
T = 
\begin{tabular}{|c|c|}
\hline
    $i_1$ & $j_1$  \\
\hline
    $i_2$ & $j_2$ \\
\hline
    \vdots & \vdots \\
\hline
    $i_k$ & $j_k$ \\
\hline
\end{tabular}\, .
\]
Without loss of generality we may assume that $i_s \le j_s$ for all $s \ge 3$. We proceed by taking cases on $s = |\{i_1, i_2, j_1, j_2 \} \cap \{ 1, \dots, \ell\}|$. 

\textbf{Case 1.} Assume $s = 0$ or $4$. We have that $i_1 < i_2$ and $j_1 < j_2$. So we may order the entries in row to obtain $T'$. Note that in this case $\Gamma_\ell$ fixes $T'$.

\textbf{Case 2.} Assume $s = 1$. Without loss of generality we assume $j_2 \in \{1, \dots, \ell \}$.
\begin{itemize}
    \item If $j_1 > i_2$ then
    \[
    \Gamma_\ell \left( \,
    \begin{tabular}{|c|c|}
    \hline
        $j_2$ & $i_1$  \\
    \hline
        $i_2$ & $j_1$ \\
    \hline
        \vdots & \vdots \\
    \hline
    \end{tabular}\,
    \right)
    =
    \begin{tabular}{|c|c|}
    \hline
        $i_1$ & $j_1$  \\
    \hline
        $i_2$ & $j_2$ \\
    \hline
        \vdots & \vdots \\
    \hline
    \end{tabular}\, .
    \]
    \item If $j_1 \le i_2$ then
    \[
    \Gamma_\ell \left( \,
    \begin{tabular}{|c|c|}
    \hline
        $j_2$ & $i_1$  \\
    \hline
        $j_1$ & $i_2$ \\
    \hline
        \vdots & \vdots \\
    \hline
    \end{tabular}\,
    \right)
    =
    \begin{tabular}{|c|c|}
    \hline
        $j_1$ & $i_1$  \\
    \hline
        $j_2$ & $i_2$ \\
    \hline
        \vdots & \vdots \\
    \hline
    \end{tabular}\, .
    \]
    The tableau on the right is row-wise equal to $T$.
\end{itemize}

\textbf{Case 3.} Assume $s = 2$. 
\begin{itemize}
    \item If $i_1, i_2 \in \{1, \dots, \ell \}$ then $\Gamma_\ell$ fixes each column of $T$, which is a semi-standard Young tableau.
    \item If $i_2, j_2 \in \{1, \dots, \ell \}$ then without loss of generality assume $i_2 \le j_2$ and $i_1 \le j_1$. We have
    \[
    \Gamma_\ell \left( \,
    \begin{tabular}{|c|c|}
    \hline
        $i_2$ & $j_2$  \\
    \hline
        $i_1$ & $j_1$ \\
    \hline
        \vdots & \vdots \\
    \hline
    \end{tabular}\,
    \right)
    =
    \begin{tabular}{|c|c|}
    \hline
        $i_1$ & $j_1$  \\
    \hline
        $i_2$ & $j_2$ \\
    \hline
        \vdots & \vdots \\
    \hline
    \end{tabular}\, .
    \]
\end{itemize}

\textbf{Case 4.} Assume $s = 3$. Without loss of generality we may assume $j_1 \in \{\ell+1, \dots, n \}$.
\begin{itemize}
    \item If $j_2 < i_1$ then
    \[
    \Gamma_\ell \left( \,
    \begin{tabular}{|c|c|}
    \hline
        $j_2$ & $i_1$  \\
    \hline
        $i_2$ & $j_1$ \\
    \hline
        \vdots & \vdots \\
    \hline
    \end{tabular}\,
    \right)
    =
    \begin{tabular}{|c|c|}
    \hline
        $i_1$ & $j_1$  \\
    \hline
        $i_2$ & $j_2$ \\
    \hline
        \vdots & \vdots \\
    \hline
    \end{tabular}\, .
    \]
    Note that in this case we have $j_2 < i_1 < i_2$ and so the tableau on the left is a semi-standard Young tableau.
    \item If $j_2 \ge i_1$ then 
    \[
    \Gamma_\ell \left( \,
    \begin{tabular}{|c|c|}
    \hline
        $i_1$ & $j_2$  \\
    \hline
        $i_2$ & $j_1$ \\
    \hline
        \vdots & \vdots \\
    \hline
    \end{tabular}\,
    \right)
    =
    \begin{tabular}{|c|c|}
    \hline
        $i_1$ & $j_1$  \\
    \hline
        $i_2$ & $j_2$ \\
    \hline
        \vdots & \vdots \\
    \hline
    \end{tabular}\, .
    \]
\end{itemize}
This completes the proof. 
\end{proof}

\begin{lemma}\label{lem:basis_rich_corresp}
Let $v,w\in{\bf I}_{k,n}$ with $v \le w$. A semi-standard Young tableau $T$ vanishes in $G_{k,n,0}|_w^v$ if and only if $\Gamma_\ell(T)$ vanishes in $G_{k,n,\ell}|_w^v$.
\end{lemma}

\begin{proof}
Let $I, J$ be the columns of $T$ and $I', J'$ be the columns of $\Gamma_\ell(T)$. The result follows from the fact that $\{\min(I), \min(J)\} = \{ \min(I'), \min(J')\}$ and similarly for the second smallest of elements of $I, J, I'$ and $J'$.
\end{proof}

By the results of Kreiman and Lakshmibai in \cite{kreiman2002richardson}, the semi-standard Young tableaux whose columns $I$ satisfy $v \le I \le w$ are a monomial basis for the Richardson variety $X_w^v$.

\begin{proposition}\label{lem:SSYT_Gamma_surj_restrict}
If $G_{k,n,\ell}|_w^v$ is monomial-free then the set
\vspace{-2mm}
\[
{\rm Im}(\Gamma_\ell)|_w^v = \{\Gamma_\ell(T) : T \textrm{ a two column semi-standard Young tableau for } X_w^v \}
\]
is a monomial basis for $\KK[P_I]_{I \in T_w^v} / \ker(\phi_\ell|_w^v)$ in degree two. 
\end{proposition}

\begin{proof}
We prove the contrapositive, i.e.~if $\textrm{Im}(\Gamma_\ell)|_w^v$ is not a monomial basis for $\KK[P_I]_{I \in T_w^v} / \ker(\phi_\ell|_w^v)$ then $G_{k,n,\ell}|_w^v$ contains a monomial. Let $T$ be a matching field tableau for $B_\ell$ representing a monomial in $\KK[P_I]_{I \in T_w^v} / G_{k,n,\ell}|_w^v$ which does not lie in the span of $\textrm{Im}(\Gamma_\ell)|_w^v$. Since $\textrm{Im}(\Gamma_\ell)$ is a basis for $G_{k, n, \ell}$, it follows that $T$ is row-wise equal to $\Gamma_\ell(T')$ for some semi-standard Young tableau $T'$ which vanishes in $G_{k,n,\ell}|_w^v$. We write $I, J$ for the columns of $T$ and $I', J'$ for the columns of $\Gamma_\ell(T')$. Since
$T$ and $\Gamma_\ell(T')$ are row-wise equal we may assume that all their entries below the second row are in semi-standard form. So we write
\vspace{-2mm}
\[
T = 
\begin{tabular}{|c|c|}
\multicolumn{1}{c}{$I$} & \multicolumn{1}{c}{$J$} \\
\hline
    $i_1$ & $j_1$  \\
\hline
    $i_2$ & $j_2$ \\
\hline
    \vdots & \vdots \\
\hline
\end{tabular}\, , \quad
\Gamma_\ell(T') = 
\begin{tabular}{|c|c|}
\multicolumn{1}{c}{$I'$} & \multicolumn{1}{c}{$J'$} \\
\hline
    $i_1$ & $j_1$  \\
\hline
    $j_2$ & $i_2$ \\
\hline
    \vdots & \vdots \\
\hline
\end{tabular}\, .
\]
Throughout the proof we write $v =\{v_1 < \dots < v_k \}$ and $w = \{w_1 < \dots < w_k \}$ for the Grassmannian permutations. We now take cases on $s = |\{i_1, i_2, j_1, j_2 \} \cap \{1, \dots, \ell \}|$.

\textbf{Case 1.} Assume $s = 0$ or $4$. It follows that $\Gamma_\ell(T')$ is a semi-standard Young tableau and so $\Gamma_\ell(T')$ does not vanish, a contradiction.

\textbf{Case 2.} Assume $s = 1$. Without loss of generality assume that $j_2 \in \{1, \dots, \ell \}$ and note that in this case we may possibly have that $I'$ and $J'$ are swapped in $\Gamma_\ell(T')$. Since $T$ does not vanish we have $v \le I, J \le w$. So by ordering the entries of $I, J$ in increasing order and comparing them with $v$ and $w$, we have
\vspace{-2mm}
\[
v_1 \le  \{i_1, j_2 \} \le w_1, \quad
v_2 \le \{ i_2, j_1 \} \le w_2.
\]
Since $\Gamma_\ell(T')$ vanishes we must have that either $I'$ or $J'$ vanishes. Let us take cases.

\textbf{Case 2.1.} Assume $I' = \{ i_1, j_2, \dots\}$ vanishes. We have    
\vspace{-2mm}
\[
v_1 \le j_2 \le w_1, \quad i_1 \le w_1 < w_2
\]
and so $I' \le w$. Since $I'$ vanishes, we must have $I' \not\ge v$ and so $i_1 < v_2$. We have the following
\begin{itemize}
    \item $v_1 \in \{1, \dots, \ell \}$ because $v_1 \le j_2$,
    \item $v_2 \in \{\ell+2, \dots, n \}$ because $v_2 > i_1 \in \{\ell+1, \dots, n \}$.
\end{itemize}
By Theorems~\ref{thm:Rich} 
we have that $G_{k,n,\ell}|_w^v$ contains a monomial.

\textbf{Case 2.2.} Assume $J' = \{ j_1, i_2, \dots \}$ vanishes. We have
\vspace{-2mm}
\[
v_1 < v_2 \le j_1, \quad
v_2 \le i_2 \le w_2.
\]
Hence $J' \ge v$. Since $J'$ vanishes we have $J \not\le w$ and so $j_1 > w_1$. We have the following
\begin{itemize}
    \item $w_i \in \{\ell+1, \dots, n \}$ for all $i \ge 2$ because $w_2 \ge i_2 \in \{\ell+1, \dots, n \}$,
    \item $w_2 \neq w_1 + 1$ because $w_1 < j_1 < i_2 \le w_2$,
    \item $w_1 \le n-k$ because $w_1 < j_1 = \min(J') \ge n-k+1$,
    \item $w_1 \ge \ell+1$ because $w_1 \ge i_1 \in \{ \ell+1, \dots, n\}$.
\end{itemize}
And so by Theorem~\ref{thm:Rich}
we have that $G_{k,n,\ell}|_w^v$ contains a monomial.

\textbf{Case 3.} Assume $s = 2$.
If $i_1, i_2 \in \{1, \dots, \ell \}$ then $\Gamma_\ell(T')$ is not a valid tableau with respect to the matching field $B_\ell$. It follows that $i_2, j_2, \in \{1, \dots, \ell\}$. However, it easily follows that $\Gamma_\ell(T')$ does not vanish in $G_{k, n, \ell}|_w^v$, a contradiction.

\textbf{Case 4.} Assume $s = 3$. Without loss of generality assume that $j_1 \in \{\ell+1, \dots, n \}$. Note that in this case we may possibly have that $I'$ and $J'$ are swapped in $\Gamma_\ell(T')$. Since $T$ does not vanish we have $v \le I, J \le w$. So by ordering the entries of $I, J$ in increasing order and comparing them with $v$ and $w$, we have
\vspace{-2mm}
\[
v_1 \le \{i_1, j_2 \} \le w_1, \quad
v_2 \le \{i_2, j_1\} \le w_2.
\]
Since $\Gamma_\ell(T')$ vanishes we must have that either $I'$ or $J'$ vanishes. We proceed by taking cases.

\textbf{Case 4.1.} Assume that $I' = \{i_1, j_2, \dots \}$ vanishes. We have
\vspace{-2mm}
\[
v_1 \le i_1 \le w_1, \quad 
j_2 \le w_1 < w_2
\]
and so $I' \le w$. Since $I'$ vanishes we must have $I' \not\ge v$ and we deduce that $j_2 < v_2$. We have the following
\begin{itemize}
    \item $v_1 \in \{1, \dots, \ell \}$ because $v_1 \le i_1 \in \{1, \dots, \ell \}$,
    \item $v_2 > v_1 + 1$ because $v_1 \le i_1 < j_2 < v_2$,
    \item $v_2 \neq \ell+1$ because $v_2 \le i_2 \in \{ 1, \dots, \ell\}$.
\end{itemize}
By Theorems~\ref{thm:Rich} 
we have that $G_{k,n,\ell}|_w^v$ contains a monomial.

\textbf{Case 4.2.} Assume that $J' = \{j_1, i_2, \dots \}$ vanishes. We have \vspace{-2mm}
\[
v_1 < v_2 \le i_2, \quad 
v_2 \le j_1 \le w_2
\]
and so $J' \ge v$. Since $J'$ vanishes we must have $J' \not\le w$ and we deduce that $i_2 > w_1$. We have: 
\begin{itemize}
    \item $w_i \in \{\ell+1, \dots, n \}$ for all $i \ge 2$ because $w_2 \ge j_1 \in \{\ell+1, \dots, n \}$,
    \item $w_2 \neq w_1 + 1$ because $w_1 < i_2 < j_1 \le w_2$,
    \item $w_1 \le n-k$ because $w_1 < i_2 = \min(J') \ge n-k+1$
    \item $w_1 \neq \ell$ because $w_1 < i_2 \in \{1, \dots, \ell \}$,
    \item $w_1 \ge 2$ because, by column $I'$, we have $i_1 < j_2 \le w_1$.
\end{itemize}
And so by Theorem~\ref{thm:Rich}
we have that $G_{k,n,\ell}|_w^v$ contains a monomial.
\end{proof}

\begin{theorem}
\label{thm:std_monomials}
If $G_{k,n,\ell}|_w^v$ is monomial-free, then the size of ${\rm Im}(\Gamma_\ell)|_w^v$ 
is equal to the number of semi-standard Young tableaux with two columns $I, J$ such that $v \le I, J \le w$. 
\end{theorem}

\begin{proof}
First note that a collection of standard monomials for $\KK[P_I]_{I \in T_w^v} / G_{k,n,0}|_w^v$ in degree two is given by semi-standard Young tableaux with two columns such that each column $I$ satisfies $v \le I \le w$. The map $\Gamma_\ell$ from Definition~\ref{def:Gamma_ell_deg2} takes each semi-standard Young tableau with two columns to a degree two monomial in $\KK[P_I]_{I \in {\bf I}_{k,n}}$. By Lemma~\ref{lem:basis_rich_corresp} we have that the restriction of $\Gamma_\ell$ to the standard monomials for $X_w^v$ gives a well-defined map to monomials in the quotient ring $\KK[P_I]_{I\in T_{w}^v} / \ker(\phi_\ell|_w^v)$. By Lemma~\ref{lem:SSYT_Gamma_injective} and Proposition~\ref{lem:SSYT_Gamma_surj_restrict} we have that the restriction of $\Gamma_\ell$ to the standard monomials of $X_w^v$ is a bijection between standard monomials for $X_w^v$ and a monomial basis for $\KK[P_I]_{I\in T_{w}^v} / \ker(\phi_\ell|_w^v)$.
\end{proof}

\vspace{-6mm}
\section{Toric degenerations of Richardson varieties}\label{sec:main_proofs}

We are now ready to answer Question~\ref{intro:question} and prove our main results, 
given the complete characterisation of monomial-free ideals $G_{k,n,\ell}|_w^v$ in \S\ref{sec:gr} and the description of a monomial basis for $\KK[P_I]_{I \in T_{w}^v} / \ker(\phi_\ell|_w^v)$ in \S\ref{sec:standard_monomial}.
In particular, we will show that when $G_{k,n,\ell}|_w^v$ is monomial-free and $\init_{\bf w_\ell}(I(X_w^v))$ is quadratically generated, then $\init_{\bf w_\ell}(I(X_w^v))$ provides a toric degeneration of the Richardson variety $X_w^v$. We expect that $\init_{\bf w_\ell}(I(X_w^v))$ is always quadratically generated, see Conjecture~\ref{conj:J_2_quad_gen},  Theorem~\ref{thm:intro_ell_0} and \cite{OllieCodeGr}. Hence, assuming that Conjecture~\ref{conj:J_2_quad_gen} holds, all the pairs $(v,w)$ classified in Theorem~\ref{thm:Rich} lead to toric degenerations of $X_w^v$.

\medskip

The first step is to prove our main results is to show that the inclusions in (\ref{eq:inclusion}) hold.

\begin{lemma}\label{lem:J_1=J_3_binomial} 
We have the following:
\begin{itemize}
    \item[{\rm (i)}] The ideals $G_{k,n,\ell}|_w^v$ and $\ker(\phi_\ell|_w^v)$ coincide if and only if $G_{k,n,\ell}|_w^v$ is monomial-free.
    \item[{\rm (ii)}] $G_{k,n,\ell}|_w^v\subseteq  {\rm in}_{{\bf w}_\ell}(I(X_w^v))$.
\end{itemize}
\end{lemma}

\begin{proof}

We first note that by Theorem~\ref{thm:JAlebra} the ideal $\init_{{\bf w}_\ell}(G_{k,n})$ is quadratically generated. We let $G$ be a quadratic binomial generating set for $\init_{{\bf w}_\ell}(G_{k,n})$.

(i) Note that $\phi_n|_w^v$ is a monomial map, hence its kernel does not contain any monomials. So, if the ideal $G_{k,n,\ell}|_w^v$ contains a monomial then it is not equal to $\ker{(\phi_\ell|_w^v)}$.
Now assume that the ideal $G_{k,n,\ell}|_w^v$ does not contain any monomials. By definition, we have $G_{k,n,\ell}|_w^v = \langle G|_{T_w^v}\rangle = \langle G \cup \{P_J : J \in {\bf I}_{k,n} \backslash T_w^v \} \rangle \cap \mathbb K[P_I]_{I \in T_w^v}$. Since the ideal $G_{k,n,\ell}|_w^v$ is monomial-free, the set $G|_{T_w^v}$ does not contain any monomials.
Moreover, since all binomials $m_1 - m_2 \in G|_{T_w^v}$ lie in the ideal $\init_{{\bf w}_\ell}(G_{k,n})$ and contain only the non-vanishing Pl\"ucker variables $P_J$ for $J \in T_w^v$, therefore, $m_1 - m_2 \in \ker{(\phi_\ell|_w^v)}$. Thus $G_{k,n,\ell}|_w^v \subseteq \ker{(\phi_\ell|_w^v)}$ which complete the proof of (i).

(ii) Since $G_{k,n,\ell}|_w^v = \langle G|_{T_w^v}\rangle$, we take $\hat g \in G|_{T_w^v}$. We have that $g \in G \subseteq \init_{{\bf w}_\ell}(G_{k,n})$ and so 
there exists $f \in G_{k,n}$ such that $\init_{{\bf w}_\ell}(f) = g$. 
The terms of $\hat g$ are precisely the non-vanishing terms of the initial terms of $f$. Thus $\hat g = \init_{\bf w}(\hat f) \in {\rm in}_{{\bf w}_\ell}(I(X_w^v))$, as desired.
\end{proof}

We now use the description of a monomial basis for the algebra $\KK[P_I]_{I\in T_{w}^v} / \ker(\phi_\ell|_w^v)$ to show that the containment in Lemma~\ref{lem:J_1=J_3_binomial}(ii) is indeed an equality.

\begin{theorem}\label{thm:main}
If $G_{k,n,\ell}|^v_w$ is monomial-free and $\init_{\bf{w}_\ell}(I(X_w^v))$ is quadratically generated, then the ideals $\init_{\bf{w}_\ell}(I(X_w^v))$, $G_{k,n,\ell}|^v_w$ and $\ker(\phi_\ell|_w^v)$ are all equal. 
\end{theorem}
\begin{proof}
By Lemma~\ref{lem:J_1=J_3_binomial} we have that $G_{k,n,\ell}|_w^v =\ker(\phi_\ell|_w^v)\subseteq {\rm in}_{{\bf w}_\ell}(I(X_w^v))$. In particular, the inclusion $\ker(\phi_\ell|_w^v)\subseteq {\rm in}_{{\bf w}_\ell}(I(X_w^v))$ implies that for any collection of monomials $\mathcal M \subseteq \KK[P_I]_{I \in T_w^v}$, if $\mathcal M$ is linearly independent in $\KK[P_I]_{I\in T_{w}^v} / {\rm in}_{{\bf w}_\ell}(I(X_w^v))$ then $\mathcal M$ is linearly independent in $\KK[P_I]_{I\in T_{w}^v} / \ker(\phi_\ell|_w^v)$. So any standard monomial basis of degree $d$ for $\KK[P_I]_{I\in T_{w}^v} / {\rm in}_{{\bf w}_\ell}(I(X_w^v))$ is linearly independent in $\KK[P_I]_{I\in T_{w}^v}/\ker(\phi_\ell|_w^v)$. Note that Gr\"obner degeneration gives rise to a flat family, and so the Hilbert polynomials of all fiber are identical. So by Theorem~\ref{thm:SMT_Grassmannian}, the dimension of degree $d$ part of $\KK[P_I]_{I\in T_{w}^v} / {\rm in}_{{\bf w}_\ell}(I(X_w^v))$ is equal to the number of standard monomials for $X_w^v$ of degree $d$. By Proposition~\ref{lem:SSYT_Gamma_surj_restrict}
we have that 
$\KK[P_I]_{I\in T_{w}^v} / {\rm in}_{{\bf w}_\ell}(I(X_w^v))$ and $\KK[P_I]_{I\in T_{w}^v} / \ker(\phi_\ell|_w^v)$ have the same number of standard monomials in degree two. Since $\ker(\phi_\ell|_w^v) \subseteq {\rm in}_{{\bf w}_\ell}(I(X_w^v))$ it follows that $\textrm{Im}(\Gamma_\ell)|_w^v$ is a collection of standard monomials for $\KK[P_I]_{I\in T_{w}^v} / {\rm in}_{{\bf w}_\ell}(I(X_w^v))$ and $G_{k,n,\ell}|_w^v$. 
Since $G_{k,n,\ell}|_w^v,{\rm in}_{{\bf w}_\ell}(I(X_w^v))$ and $\ker(\phi_\ell|_w^v)$ are all generated in degree two, it follows that they are equal.
\end{proof}

\begin{theorem}
\label{thm:intro_ell_0}
For $\ell=0$, the ideals $G_{k,n,0}|_w^v$, $\init_{{\bf w}_0}(I(X_w^v))$ and $\ker(\phi_0|_w^v)$ are all equal. In particular, they are all quadratically generated toric ideals.
\end{theorem}\begin{proof}
First note that by Theorem~\ref{thm:Rich} each ideal $G_{k,n,0}|_w^v$ is monomial-free. By Theorem~\ref{thm:SMT_Grassmannian} the standard monomials for the Pl\"ucker algebra of the Richardson variety are in bijection with the rectangular semi-standard Young tableaux with columns $J$ such that $v \le J \le w$. 
To see that these tableaux form a monomial basis for $\KK[P_I]_{I\in T_{w}^v}/\ker(\phi_0|_w^v)$, observe that any two monomials in this algebra are equal if and only if their corresponding tableaux are row-wise equal. Also, for any tableau, there is a unique semi-standard Young tableau which is row-wise equal to it.
It follows that the dimension of the degree $d$ part of $\KK[P_I]_{I\in T_{w}^v}/\ker(\phi_0|_w^v)$ and $\KK[P_I]_{I\in T_{w}^v}/\init_{{\bf w}_0}(I(X_w^v))$ are equal for all $d$. Moreover, by Lemma~\ref{lem:J_1=J_3_binomial} we have that $\ker(\phi_0|_w^v) = G_{k,n,0}|_w^v \subseteq \init_{{\bf w}_0}(I(X_w^v))$. Hence, these ideals are all equal and, in particular, they are quadratically generated. 
\end{proof}

\begin{remark}\label{rem:computation}
In  \cite{OllieCodeGr}, using \texttt{Macaulay2} we have calculated the initial ideals ${\rm in}_{{\bf w}_\ell}(I(X_w^v))$ for all Richardson varieties inside $G_{k,n}$ where $n \leq 7$ and $k\leq n-2$. 
We have observed that they are all quadratically generated. This 
confirms that Conjecture~\ref{conj:J_2_quad_gen} holds for $n\leq 7$.
\end{remark}

We proceed by comparing our results to previous results in the literature. We highlight possible connections to other areas and future research directions.

\begin{remark}
In \cite{bossinger2018following}, the authors study the degenerations of Schubert varieties inside the full flag variety built upon on the flat degeneration given by Feigin \cite{Feigin2012}. They give a number of sufficient conditions on the permutation $w \in S_n$ such that the restriction of the degeneration to the Schubert variety $X_w$ is reducible. Similarly to our methods, this is done by showing that the corresponding initial ideals contain monomials.  
In \cite{CharyOllieFatemeh_WICA}, we use the results of \cite{OllieFatemeh2} to study the degenerations of Richardson varieties inside the flag variety.  
\end{remark}

\begin{remark}
In \cite{clarke2020combinatorial}, the authors showed that the polytopes associated to toric degenerations of the Grassmannian arising from matching fields, are related to each other by sequences of combinatorial mutations in the sense of \cite{akhtar2012minkowski}. We expect that the polytopes of toric degenerations of Richardson varieties provided here to have similar properties.
\end{remark}

\vspace{-3mm}

\bibliographystyle{plain}
\bibliography{JACO-Trop1.bib}

\bigskip
\bigskip
\noindent
\footnotesize{\bf Authors' addresses:}

\medskip

\noindent Narasimha Chary Bonala\\ Ruhr-Universit\"at Bochum,
Fakult\"at f\"ur Mathematik, D-44780 Bochum, Germany
\\
\noindent  E-mail address: {\tt  narasimha.bonala@rub.de}
\medskip

\noindent Oliver Clarke\\ University of Bristol, School of Mathematics,
BS8 1TW, Bristol, UK
\\
\noindent  E-mail address: {\tt oliver.clarke@bristol.ac.uk}

\medskip

\noindent Fatemeh Mohammadi \\
Department of Mathematics: Algebra and Geometry, Ghent University, 9000 Ghent, Belgium \\
Department of Mathematics and Statistics, 
UiT – The Arctic University of Norway, 9037 Troms\o, Norway
\\ E-mail address: {\tt fatemeh.mohammadi@ugent.be}

\end{document}